\documentclass[reqno, 11pt]{amsart} 
\usepackage{amsfonts, amsmath, amssymb, amsthm}
\usepackage[margin=2.75cm, heightrounded]{geometry}
\usepackage{fancybox}
\usepackage{hhline, float}
\usepackage{mathrsfs}
\usepackage{nccmath}
\usepackage[dvipsnames]{xcolor}
\usepackage[colorlinks=true]{hyperref}
\hypersetup{citecolor=red, linkcolor=blue}

\allowdisplaybreaks
\numberwithin{equation}{section}

\usepackage[dvipsnames]{xcolor}
\usepackage[colorlinks=true]{hyperref}
\hypersetup{citecolor=red, linkcolor=blue}

\usepackage{amsthm}

\newtheorem{theorem}{Theorem}[section]

\newtheorem{lemma}[theorem]{Lemma}

\newtheorem*{theoremA}{Theorem A}

\theoremstyle{definition}

\newtheorem{remark}[theorem]{Remark}

\makeatletter
\@namedef{subjclassname@2020}{%
\textup{2020} Mathematics Subject Classification}
\makeatother

\begin{document}

\title[Non-degeneracy and local uniqueness of solutions]
{Non-degeneracy and local uniqueness of bubbling solutions for a singular Liouville equation}

\author{Yuxia Guo}
\address[Yuxia Guo]{Department of Mathematical Sciences, Tsinghua University}
\email{yguo@tsinghua.edu.cn}

\author{Jionghao Lv}
\address[Jionghao Lv]{Department of Mathematical Sciences, Tsinghua University}
\email{lvjh26@mails.tsinghua.edu.cn}

\begin{abstract}
We consider the following singular Liouville equation
\[
-\Delta v=\lambda V(x)|x|^2e^v \quad \text{in } B_1,
\quad
v=0 \quad \text{on } \partial B_1,
\]
where \(B_1\subset\mathbb R^2\) is the unit disk, \(\lambda>0\) is a small parameter, and \(V\) is a positive smooth function. We first prove a non-degeneracy result for the bubbling solutions constructed in \cite{D-W-Z2026} by the local Pohozaev identities. Then we use this non-degeneracy result to establish the local uniqueness of bubbling solutions whose concentration parameters are sufficiently close.
\end{abstract}
\maketitle


\section{Introduction}
Liouville-type equations with exponential nonlinearities arise naturally in two-dimensional nonlinear elliptic theory, with motivations ranging from geometry to mathematical physics. Typical examples include prescribed Gaussian curvature equations \cite{K-W1974}, conformal geometry \cite{A1998}, statistical mechanics \cite{C-L-M-P1992,C-L-M-P1995}, Chern-Simons-Higgs theory \cite{H-K-P1990,J-W1990}, electroweak theory \cite{A-O1990} and the Toda system
\cite{J-W2001}.

From a mathematical perspective, a fundamental difficulty in the study of Liouville type equations is the possible loss of compactness caused by exponential nonlinearity. The loss of compactness is often manifested through the blow-up phenomenon, whose analysis is closely related to results on existence, compactness, a priori estimates, etc. The blow-up behavior at points away from the origin has been thoroughly investigated. Starting with the uniform estimates of Brezis and Merle \cite{B-M1991}, and the quantization result of Li and Shafrir \cite{L-S1994}, refined blow-up analysis and asymptotic estimates were further established in \cite{C-L2002,C-L2003,L-1999,Z2006}.

In this paper, we are concerned with the following singular Liouville equation
\begin{equation*}
    -\Delta v = \lambda |x|^{2N} V(x)e^v
\end{equation*}
in a planar domain containing the origin, where $\lambda>0$ is a small parameter, $V$ is a positive smooth function, and $N>0$. When the blow-up point is away from the singular source, after a suitable scaling, the leading-order profile is given by the standard Liouville bubble. The situation becomes more delicate when blow-up occurs at the singular point. In particular, when $N\in\mathbb{N}$, the singular source is quantized and the limiting problem
\[
-\Delta U = |x|^{2N} e^U,
\quad
\int_{\mathbb{R}^2} |x|^{2N} e^U < +\infty
\]
admits non-radial finite mass solutions. By the classification result of Prajapat and Tarantello \cite{P-T2001}, viewing $\mathbb{R}^{2}$ as $\mathbb{C}$, such solutions are given by
\begin{equation}\label{eq:bubble U tau b}
    U_{\tau,b}
= \log \frac{8(N+1)^2 \tau}
{(\tau + \left|x^{N+1}-b\right|^2)^2},
\quad \tau > 0,\ b \in \mathbb{C},
\end{equation}
where $x^{N+1}$ denotes the complex $(N+1)$-power. This additional parameter $b$ is responsible for the possible occurrence of non-simple blow-up, which is equivalent to saying that the bubbling solutions may not satisfy the spherical Harnack inequality, and multiple local maxima near the singular source may appear. A lot of results have  been devoted to the analysis of such non-simple blow-up phenomena. We refer to \cite{K-L2016,B-T2002,D-W-Z2022,D-W-Z2024cvpde,D-W-Z2024,D-W-Z2026,W-Z2022,W-Z2026} and references therein. The obtained results revealed that the coefficient function must satisfy some vanishing conditions if non-simple blow-up occurs as a quantized singularity. Wei and Zhang \cite{W-Z2022} proved the first-order vanishing estimates  and later a Laplacian vanishing theorem in \cite{W-Z2026}. These results were further strengthened by D'Aprile, Wei and Zhang, who established higher-order vanishing theorems for non-simple blow-up solutions \cite{D-W-Z2024}. These vanishing results show that non-simple blow-up is strongly influenced by the behavior of the coefficient function near the singular source.

In the present paper, we focus on the case \(N=1\) and take \(\Omega\) to be the unit disk \(B_1\subset\mathbb{R}^{2}\) centered at the origin. More precisely, for a sequence $\lambda_m\to0^{+}$, let $v_m$ be a corresponding sequence of solutions to the following Dirichlet problem
\begin{equation}\label{eq:main problem}
\begin{cases}
-\Delta v_m = \lambda_m V(x)|x|^2 e^{v_m} & \text{in } B_1,\\
v_m = 0 & \text{on } \partial B_1.
\end{cases}
\end{equation}
We assume that the total mass is uniformly bounded, namely, there exists $C>0$, independent of $m$, such that
\[
\lambda_m \int_{B_1} V(x)|x|^2 e^{v_m}\,dx \le C,
\]
and that $v_m$ admits the origin as its only blow-up point in \(B_1\):
\begin{equation}\label{eq:0 is unique blowup point}
\max_{B_1} v_m \to +\infty,
\quad
\sup_{\varepsilon \leq |x| \leq 1} v_m < C(\varepsilon)
\quad \forall \varepsilon > 0.
\end{equation}
Furthermore, we impose standard assumptions on \(V\):
\begin{equation}\label{eq:standard assumptions on V}
V \in C^3(\overline{B_1}),
\quad
\min_{\overline{B_1}} V > 0,
\end{equation}
and
\begin{equation}\label{eq:V0 nabla V0 and nondegenerate}
    V(0)=1,
\quad
\nabla V(0)=0.
\end{equation}
We also assume that $0$ is a nondegenerate critical point of $V$. Hence, up to a rotation of the coordinates, \(V\) admits the following expansion in a small neighborhood of the origin:
\begin{equation}\label{eq:V}
V(x) = 1 + \frac{\gamma_1 x_1^2 + \gamma_2 x_2^2}{2} + O(|x|^3)
\quad \text{as } x \to 0.
\end{equation}
The result obtained in \cite{D-W-Z2026} states the following:
\begin{theoremA}\label{th A}
Assume that $V$ satisfies \eqref{eq:standard assumptions on V}-\eqref{eq:V} and
\begin{equation}\label{eq:hessian V is nondegenerate}
\det D^2 V(0) = \gamma_1 \cdot \gamma_2 > 0.
\end{equation}
Then there is an integer $m_0>0$ such that for any integer $m\geq m_0$, problem \eqref{eq:main problem}-\eqref{eq:0 is unique blowup point} has a solution $v_m$ of the form
\begin{equation}\label{eq:vm=PWm+phi m}
    v_m = P W_{m} + \phi_m,
\end{equation}
where $PW_m$ denotes the projection of $W_m$ onto $H_0^{1}(B_1)$, namely,
\begin{equation*}
\begin{cases}
\Delta P W_{m}=\Delta W_{m}, & x\in B_1,\\[4pt]
P W_{m}=0, & x\in \partial B_1,
\end{cases}
\end{equation*}
and
\begin{equation}\label{eq:bubble}
    W_m = W_{\lambda_m,b_m}
= \log \frac{\lambda_{m}}
{\left( \frac{\lambda_{m}}{32} + |x^2-b_m|^2 \right)^2},
\end{equation}
$\phi_m\in H_{0}^{1}(B_1)$, and as $m\to+\infty$, $\lambda_m\to0^{+}$, $b_{m}=O(\sqrt{\lambda_m})$,
\begin{equation}\label{eq:phim H01 estimate}
    \|\phi_m\|_{H_{0}^{1}(B_1)}\leq \lambda_{m}^{\frac{1}{2}-\varepsilon}
\end{equation}
for any fixed $\varepsilon>0$. 

Furthermore, setting 
\[
    \rho_{m}:=\left(\frac{\lambda_{m}}{32}\right)^{\frac{1}{4}},\quad\xi_m:=\frac{b_m}{\rho_m^{2}},
\]
one has $\xi_m\rightarrow\xi=(\xi^{*},0)\in \mathbb{R}^{2}$. After a suitable scaling, the limiting profile of $v_m$ is a bubble \eqref{eq:bubble U tau b} with $N=1$ and $b=\xi\in\mathbb{R}^{2}$, where
\[
\xi^{*}>0\ \text{if}\ |\gamma_1|<|\gamma_2|,\quad
\xi^{*}<0\ \text{if}\ |\gamma_1|>|\gamma_2|,\quad
\xi^{*}=0\ \text{if}\ \gamma_1=\gamma_2.
\]

\end{theoremA}
\begin{remark}
The statement of Theorem A is quoted with a correction. In the original paper, the estimate corresponding to \eqref{eq:phim H01 estimate} contains the
exponent $\frac{3}{4}-\varepsilon$ due to a minor computational slip; the correct exponent should be $\frac{1}{2}-\varepsilon$.
\end{remark}

The purpose of this paper is to study the linearized structure of these bubbling solutions. Our first main result is the following.
\begin{theorem}\label{thm nondegenerate}
    Suppose that $V$ satisfies \eqref{eq:standard assumptions on V}-\eqref{eq:hessian V is nondegenerate}. Then there exists an integer $m_1>0$, such that for any integer $m>m_1$, the solution $v_m$ obtained in Theorem A is non-degenerate in the sense that if $\psi\in H_{0}^{1}(B_1)$ is a solution of the following equation:
\[
\begin{cases}
\widehat{L}_m\psi:=-\Delta \psi-\lambda_m V(x)|x|^2e^{v_m}\psi=0,
& x\in B_1,\\[4pt]
\psi=0,
& x\in \partial B_1,
\end{cases}
\]
then $\psi=0$.
\end{theorem}

Our second result is a local uniqueness theorem. Roughly speaking, if two bubbling solutions have the same parameter $\lambda_m$, satisfy the same type of a priori estimates as the constructed solutions, and their concentration parameters are sufficiently close, then the two solutions must coincide.
\begin{theorem}\label{thm:local uniqueness}
    Assume that $V$ satisfies the assumptions in Theorem \ref{thm nondegenerate}. Let $v_m^{(1)}$ and $v_m^{(2)}$ be two sequences of
solutions to \eqref{eq:main problem} with the same parameter $\lambda_m$, satisfying
\[
v_m^{(i)}
=
PW_{\lambda_m,b_m^{(i)}}+\phi_m^{(i)},
\quad i=1,2,
\]
and, for some fixed $M>0$ and sufficiently small $\varepsilon>0$
\begin{equation}\label{eq:local uniqueness assumption}
    |b_m^{(i)}|\le M\sqrt{\lambda_m},
\quad
\|\phi_m^{(i)}\|_{H_{0}^{1}(B_1)}\leq \lambda_{m}^{\frac{1}{2}-\varepsilon},
\qquad i=1,2.
\end{equation}
Moreover, for some $\sigma>\frac14$,
\begin{equation}\label{eq:bm1-bm2}
    |b_m^{(1)}-b_m^{(2)}|
=
O(\lambda_m^{\frac12+\sigma}).
\end{equation}
Then there is an integer $m_2>0$ such that, for every $m\ge m_2$, $
v_m^{(1)}\equiv v_m^{(2)}$.
\end{theorem}

The proof of non-degeneracy proceeds by contradiction. Starting from a normalized kernel element, we rescale around the blow-up point and obtain a solution of the limiting linearized equation in $\mathbb{R}^{2}$. The kernel of the limiting operator is generated by three functions: one scaling mode and two translation modes. We then decompose the original kernel element into the two projected parameter modes and an orthogonal remainder. A priori estimates shows that the orthogonal remainder is controlled by the two corresponding coefficients. Finally, two Pohozaev-type identities and non-degeneracy condition \eqref{eq:hessian V is nondegenerate} rule out the remaining two coefficients. The estimate for the orthogonal part then forces the whole kernel element to vanish, contradicting the normalization.

The proof of local uniqueness follows a similar strategy, applied to the normalized difference of two solutions. In this case, the linearized coefficient is an averaged exponential term. The closeness assumption \eqref{eq:bm1-bm2} ensures that the perturbation caused by the gap between the two concentration parameters is sufficiently small and can be absorbed in the estimates.

The paper is organized as follows. Section 2 contains preliminaries, where we collect several lemmas that will be used later. In Section 3, we prove the non-degeneracy of the bubbling solutions. In Section 4, we establish the local uniqueness result.

Throughout the paper, we use $C$ to denote a positive constant which may vary from line to line.
\section{Preliminaries}

We begin by recalling the notation and some basic estimates from \cite{D-W-Z2026}. Let \(\mathcal H(x,y)\) be the regular part of the Green function in \(B_1\):
\begin{equation}\label{H(x,y)}
\mathcal H(x,y)
=
\frac{1}{2\pi}
\log\left( |x| \left| y-\frac{x}{|x|^2} \right| \right)
=
\frac{1}{4\pi}
\log\left( |x|^2 |y|^2 + 1 - 2\langle x,y\rangle \right),
\end{equation}
where $\langle\cdot,\cdot\rangle$ denotes the inner product in $\mathbb R^2$. If $b_m$ is sufficiently close to $0$,  the function $\mathcal H(x^2,b_m)$ is harmonic in $B_{1}$ and satisfies
\[
\mathcal H(x^2,b_m)
=
\frac{1}{2\pi}\log |x^2-b_m|
\quad \text{on } \partial B_1.
\]
Then, by the definition of $P$ and the maximum principle, one obtains the expansion
\begin{equation}\label{eq:Pw expand}
PW_m
=
W_m - \log\lambda_m
+8\pi \mathcal H(x^2,b_m)
+O(\lambda_m)
\end{equation}
uniformly for $x\in\overline{B_1}$ and \(b_m\) in a small neighborhood of \(0\). Equivalently,
\begin{equation*}
    PW_m=-2\log\left(
\frac{\lambda_m}{32}+|x^2-b_m|^2
\right)
+8\pi \mathcal H(x^2,b_m)
+O(\lambda_m).
\end{equation*}

The following lemma is a special case of the non-degeneracy result for entire solutions of the singular Liouville equation.
\begin{lemma}\cite{D-E-M2012}\label{lem:nondegeneracy of finite mass solutions of the singular Liouville equation}
Let $c\in\mathbb{C}$. Assume that $\phi:\mathbb{R}^2\to\mathbb{R}$ solves the problem
\begin{equation*}
-\Delta \phi
=
32\frac{|y|^2}{\left(1+|y^2-c|^2\right)^2}\phi
\quad \textit{in } \mathbb{R}^2,
\quad
\int_{\mathbb{R}^2} |\nabla \phi(y)|^2\,dy < +\infty .
\end{equation*}
Then there exist $a_{0}$, $a_{1}$, $a_{2}\in\mathbb{R}$ such that
\[
\phi(y)=a_0Z_0+a_1Z_1+a_2Z_2
\]
where
\begin{equation}\label{eq:Z012}
    Z_0(y):=
\frac{1-|y^2-c|^2}{1+|y^2-c|^2},
\quad
Z_1(y):=
\frac{\operatorname{Re}(y^2-c)}{1+|y^2-c|^2},
\quad
Z_2(y):=
\frac{\operatorname{Im}(y^2-c)}{1+|y^2-c|^2}.
\end{equation}
\end{lemma}

We next analyze the linearized operator, following the strategy of
\cite{D-W-Z2026}. More precisely, after the change
of variables $x=\left(\frac{\lambda_{m}}{32}\right)^{\frac14}y,$ Lemma~\ref{lem:nondegeneracy of finite mass solutions of the singular Liouville equation} implies that all solutions $\psi$
of
\[
-\Delta\psi
=
\frac{\lambda_{m} |x|^2}
{\left(\frac{\lambda_{m}}{32}+|x^2-b_m|^2\right)^2}\psi
=
|x|^2 e^{W_m}\psi
\quad \text{in } \mathbb{R}^2,
\quad
\int_{\mathbb{R}^2} |\nabla\psi|^2\,dx<+\infty
\]
are linear combinations of
\begin{equation}\label{eq: definition of Zm012}
 Z_m ^0(x)
=
\frac{\frac{\lambda_{m}}{32}-|x^2-b_{m}|^2}
{\frac{\lambda_{m}}{32}+|x^2-b_m|^2},\,Z_m^1(x)
=
\frac{\sqrt{\frac{\lambda_m}{32}}\,\operatorname{Re}(x^2-b_m)}
{\frac{\lambda_m}{32}+|x^2-b_m|^2},\,Z_m^2(x)
=
\frac{\sqrt{\frac{\lambda_m}{32}}\,\operatorname{Im}(x^2-b_m)}
{\frac{\lambda_m}{32}+|x^2-b_m|^2}.
\end{equation}
As in \cite{D-W-Z2026}, we introduce the projections
$PZ_m^j$ of $Z_m^j$ onto $H_0^1(B_1)$, then
\begin{equation}\label{eq:PZm0}
PZ_m^0(x)
=
Z_m^0(x)+1+O(\lambda_m),
\end{equation}
and
\begin{equation}\label{eq:PZm12}
PZ_m^j(x)
=
Z_m^j(x)+O(\sqrt{\lambda_m}),
\quad j=1,2,
\end{equation}
uniformly with respect to $x\in \overline{B_1}$ and $b_m$ in a small
neighborhood of $0$.

In the rest of the paper, we adopt the following notations. Let $\|\cdot\|$
and $\|\cdot\|_p$ denote the norms on $H_{0}^{1}(B_1)$ and $L^{p}(B_1)$, respectively, i.e.,
\begin{equation*}
    \|u\|:=\|u\|_{H_{0}^{1}(B_1)},\quad \|u\|_{p}:=\|u\|_{L^{p}(B_{1})},\quad \forall u\in H_{0}^{1}(B_1).
\end{equation*}

We recall the well-known Moser-Trudinger inequality (\cite{M1970,T1967}).

\begin{lemma}\label{lem:M-T inequality}
There exists $C>0$ such that
\begin{align*}
\int_{B_1} e^{\frac{4\pi u^2}{\|u\|^2}}\,dy
\leq C,
\quad \forall u \in H_0^1(B_1).
\end{align*}
In particular, for any $q \geq 1$,
\begin{align*}
\|e^u\|_q
\leq C_q e^{\frac{q}{16\pi}\|u\|^2},
\quad \forall u \in H_0^1(B_1).
\end{align*}
\end{lemma}

We also need the following estimate from \cite{D-W-Z2026}.
\begin{lemma}\label{lem:zhangLp}
Let $s=0,1,2,3$ and $p\geq 1$ be fixed. The following holds:
\begin{equation}\label{eq:Lp estimate in Zhang}
\bigl\||x|^{2+s}e^{W_m}\bigr\|_{p} \leq C\lambda_m^{\frac{s}{4}}\lambda_m^{-\frac{p-1}{2p}},\quad \bigl\|\lambda_m |x|^{2+s}e^{PW_m}\bigr\|_{p} \leq C\lambda_m^{\frac{s}{4}}\lambda_m^{-\frac{p-1}{2p}}
\end{equation}
 uniformly for $|b_m|\leq M\sqrt{\lambda_m}$, where $M>0$ is a sufficiently large number to be chosen later.
\end{lemma}

\section{The non-degeneracy of the solutions}

Let
\begin{equation}\label{eq:main-equation}
\begin{cases}
-\Delta v
=
\lambda V(x)|x|^2 e^{v},
& x\in B_1,\\[2mm]
v=0,
& x\in \partial B_1,
\end{cases}
\end{equation}
and
\begin{equation}\label{eq:linearized-equation}
\begin{cases}
-\Delta \psi
=
\lambda V(x)|x|^2 e^{v}\psi,
& x\in B_1,\\[2mm]
\psi=0,
& x\in \partial B_1.
\end{cases}
\end{equation}

We have the following identities.
\begin{lemma}\label{lem:Pohozaev identity}
It holds
\begin{equation}\label{eq:Pohozaev one}
\begin{aligned}
&\lambda\int_{B_1} e^{v}\psi
\left(
\frac{\partial V}{\partial x_1}x_1
-
\frac{\partial V}{\partial x_2}x_2
\right)\,dx  =\\
&\quad
\int_{\partial B_1}
\frac{\partial v}{\partial \nu}
\frac{\partial \psi}{\partial \nu}
(x_1^2-x_2^2)\,d\sigma
-2\pi
\frac{\partial v}{\partial x_1}(0)
\frac{\partial \psi}{\partial x_1}(0)
+
2\pi
\frac{\partial v}{\partial x_2}(0)
\frac{\partial \psi}{\partial x_2}(0),
\end{aligned}
\end{equation}
and
\begin{equation}\label{eq:Pohozaev two}
\begin{aligned}
&\lambda\int_{B_1}e^{v}\psi
\left(
\frac{\partial V}{\partial x_1}x_2
+
\frac{\partial V}{\partial x_2}x_1
\right)\,dx =\\
&\quad
2\int_{\partial B_1}
\frac{\partial v}{\partial \nu}
\frac{\partial \psi}{\partial \nu}
x_1x_2\,d\sigma
-2\pi
\left[
\frac{\partial v}{\partial x_1}(0)
\frac{\partial \psi}{\partial x_2}(0)
+
\frac{\partial v}{\partial x_2}(0)
\frac{\partial \psi}{\partial x_1}(0)
\right].
\end{aligned}
\end{equation}
\end{lemma}
\begin{proof}
We prove the two identities by multiplying the equations with suitable test functions.
For $\varepsilon>0$ sufficiently small, set
\[
\Omega_\varepsilon:=B_1\setminus {B_\varepsilon}.
\]

We first prove \eqref{eq:Pohozaev one}. Define
\begin{equation}\label{eq:X(x)}
    X(x):=
\left(
\frac{x_1}{|x|^2},
-\frac{x_2}{|x|^2}
\right),
\quad x\neq0.
\end{equation}
Multiplying \eqref{eq:main-equation} by $X\cdot \nabla\psi$ and \eqref{eq:linearized-equation}
by $X\cdot \nabla v$ and integrating on $\Omega_{\varepsilon}$, we derive
\begin{equation}\label{eq:two equation time X dot nabla and sum}
\int_{\Omega_\varepsilon}
\left[
(-\Delta v)(X\cdot\nabla\psi)
+
(-\Delta\psi)(X\cdot\nabla v)
\right]\,dx
=
\lambda\int_{\Omega_\varepsilon}
V|x|^2e^v
\left[
X\cdot\nabla\psi
+
\psi X\cdot\nabla v
\right]\,dx .
\end{equation}
Integrating the left hand side by parts, we get
\begin{equation}\label{eq:int Omega e for Delta v}
    \begin{aligned}
\int_{\Omega_\varepsilon}
(-\Delta v)(X\cdot\nabla\psi)\,dx
&=
-\int_{\partial\Omega_\varepsilon}
\frac{\partial v}{\partial\nu}\, X\cdot\nabla\psi\,d\sigma
+
\int_{\Omega_\varepsilon}
\nabla v\cdot\nabla(X\cdot\nabla\psi)\,dx .
\end{aligned}
\end{equation}
Likewise,
\begin{equation}\label{eq:int Omega e for Delta psi}
\begin{aligned}
\int_{\Omega_\varepsilon}
(-\Delta\psi)(X\cdot\nabla v)\,dx
&=
-\int_{\partial\Omega_\varepsilon}
\frac{\partial\psi}{\partial\nu}\, X\cdot\nabla v\,d\sigma
+
\int_{\Omega_\varepsilon}
\nabla\psi\cdot\nabla(X\cdot\nabla v)\,dx .
\end{aligned}
\end{equation}
We proceed to expand the interior term. It holds that
\begin{equation*}
    \begin{aligned}
\nabla v\cdot\nabla(X\cdot\nabla\psi)
=
\frac{\partial v}{\partial x_1}
\frac{\partial}{\partial x_1}
\left(
X_1\frac{\partial\psi}{\partial x_1}
+
X_2\frac{\partial\psi}{\partial x_2}
\right)
+
\frac{\partial v}{\partial x_2}
\frac{\partial}{\partial x_2}
\left(
X_1\frac{\partial\psi}{\partial x_1}
+
X_2\frac{\partial\psi}{\partial x_2}
\right),
\end{aligned}
\end{equation*}
and
\[
\begin{aligned}
\nabla\psi\cdot\nabla(X\cdot\nabla v)=
\frac{\partial \psi}{\partial x_1}
\frac{\partial}{\partial x_1}
\left(
X_1\frac{\partial v}{\partial x_1}
+
X_2\frac{\partial v}{\partial x_2}
\right)
+
\frac{\partial \psi}{\partial x_2}
\frac{\partial}{\partial x_2}
\left(
X_1\frac{\partial v}{\partial x_1}
+
X_2\frac{\partial v}{\partial x_2}
\right).
\end{aligned}
\]
Therefore,
\begin{equation}\label{eq:expand nabla}
    \begin{aligned}
&\quad\nabla v\cdot\nabla(X\cdot\nabla\psi)
+
\nabla\psi\cdot\nabla(X\cdot\nabla v) \\
&=
2\frac{\partial X_1}{\partial x_1}
\frac{\partial v}{\partial x_1}
\frac{\partial\psi}{\partial x_1}+2\frac{\partial X_2}{\partial x_2}
\frac{\partial v}{\partial x_2}
\frac{\partial\psi}{\partial x_2}  
+
\left(
\frac{\partial X_2}{\partial x_1}
+
\frac{\partial X_1}{\partial x_2}
\right)
\left(
\frac{\partial v}{\partial x_1}
\frac{\partial\psi}{\partial x_2}
+
\frac{\partial v}{\partial x_2}
\frac{\partial\psi}{\partial x_1}
\right) \\
&\quad
+
X_1\frac{\partial}{\partial x_1}
\left(
\frac{\partial v}{\partial x_1}
\frac{\partial\psi}{\partial x_1}
+
\frac{\partial v}{\partial x_2}
\frac{\partial\psi}{\partial x_2}
\right) 
+
X_2\frac{\partial}{\partial x_2}
\left(
\frac{\partial v}{\partial x_1}
\frac{\partial\psi}{\partial x_1}
+
\frac{\partial v}{\partial x_2}
\frac{\partial\psi}{\partial x_2}
\right).
\end{aligned}
\end{equation}
Notice that
\[
\frac{\partial v}{\partial x_1}
\frac{\partial\psi}{\partial x_1}
+
\frac{\partial v}{\partial x_2}
\frac{\partial\psi}{\partial x_2}
=
\nabla v\cdot\nabla\psi .
\]
Hence,
\begin{equation}\label{eq:X dot nabla nabla v dot nabla psi}
\begin{aligned}
X_1\frac{\partial}{\partial x_1}
\left(
\frac{\partial v}{\partial x_1}
\frac{\partial\psi}{\partial x_1}
+
\frac{\partial v}{\partial x_2}
\frac{\partial\psi}{\partial x_2}
\right)
+
X_2\frac{\partial}{\partial x_2}
\left(
\frac{\partial v}{\partial x_1}
\frac{\partial\psi}{\partial x_1}
+
\frac{\partial v}{\partial x_2}
\frac{\partial\psi}{\partial x_2}
\right) =
X\cdot\nabla(\nabla v\cdot\nabla\psi).
\end{aligned}
\end{equation}
Integrating \eqref{eq:X dot nabla nabla v dot nabla psi} by parts on $\Omega_{\varepsilon}$, we get
\begin{equation}\label{eq: int Omega e for X dot nabla(nabla v dot nabla psi)}
\int_{\Omega_\varepsilon}
X\cdot\nabla(\nabla v\cdot\nabla\psi)\,dx
=
\int_{\partial\Omega_\varepsilon}
(\nabla v\cdot\nabla\psi)X\cdot\nu\,d\sigma
-\int_{\Omega_\varepsilon}
(\nabla v\cdot\nabla\psi)\operatorname{div}X\,dx .
\end{equation}

By \eqref{eq:int Omega e for Delta v}, \eqref{eq:int Omega e for Delta psi}, \eqref{eq:expand nabla} and \eqref{eq: int Omega e for X dot nabla(nabla v dot nabla psi)}, we obtain
\begin{equation}\label{int Delta v + Delta psi}
\begin{aligned}
&\quad\quad\int_{\Omega_\varepsilon}
\left[
(-\Delta v)(X\cdot\nabla\psi)
+
(-\Delta\psi)(X\cdot\nabla v)
\right]\,dx \\
&=
-\int_{\partial\Omega_\varepsilon}
\frac{\partial v}{\partial\nu}\, X\cdot\nabla\psi\,d\sigma
-\int_{\partial\Omega_\varepsilon}
\frac{\partial\psi}{\partial\nu}\, X\cdot\nabla v\,d\sigma \\
&\quad
+
\int_{\partial\Omega_\varepsilon}
(\nabla v\cdot\nabla\psi)X\cdot\nu\,d\sigma \\
&\quad
+
\int_{\Omega_\varepsilon}
\bigg[
2\frac{\partial X_1}{\partial x_1}
\frac{\partial v}{\partial x_1}
\frac{\partial\psi}{\partial x_1} 
+
\left(
\frac{\partial X_2}{\partial x_1}
+
\frac{\partial X_1}{\partial x_2}
\right)
\left(
\frac{\partial v}{\partial x_1}
\frac{\partial\psi}{\partial x_2}
+
\frac{\partial v}{\partial x_2}
\frac{\partial\psi}{\partial x_1}
\right) \\
&\qquad
+
2\frac{\partial X_2}{\partial x_2}
\frac{\partial v}{\partial x_2}
\frac{\partial\psi}{\partial x_2} 
-
\left(
\frac{\partial X_1}{\partial x_1}
+
\frac{\partial X_2}{\partial x_2}
\right)
\left(
\frac{\partial v}{\partial x_1}
\frac{\partial\psi}{\partial x_1}
+
\frac{\partial v}{\partial x_2}
\frac{\partial\psi}{\partial x_2}
\right)
\bigg]\,dx .
\end{aligned}
\end{equation}

Recall that $X(x)$ is defined in \eqref{eq:X(x)}, a direct computation shows that
\begin{equation}\label{partian X1 and X2}
\frac{\partial X_1}{\partial x_1}
=
\frac{\partial X_2}{\partial x_2}
=
\frac{x_2^2-x_1^2}{|x|^4},
\quad
\frac{\partial X_2}{\partial x_1}
=
-\frac{\partial X_1}{\partial x_2}
=
\frac{2x_1x_2}{|x|^4}.
\end{equation}
Substituting \eqref{partian X1 and X2} into \eqref{int Delta v + Delta psi}, we have
\begin{equation}\label{simplify int Delta v + Delta psi}
\begin{aligned}
&\int_{\Omega_\varepsilon}
\left[
(-\Delta v)(X\cdot\nabla\psi)
+
(-\Delta\psi)(X\cdot\nabla v)
\right]\,dx  \\
&=
-\int_{\partial\Omega_\varepsilon}
\frac{\partial v}{\partial\nu}\,X\cdot\nabla\psi\,d\sigma
-\int_{\partial\Omega_\varepsilon}
\frac{\partial\psi}{\partial\nu}\,X\cdot\nabla v\,d\sigma
+\int_{\partial\Omega_\varepsilon}
(\nabla v\cdot\nabla\psi)\,X\cdot\nu\,d\sigma .
\end{aligned}
\end{equation}
Next let us examine separately the boundary integrals over $\partial B_{1}$ and over $\partial B_{\varepsilon}$: taking into account that the homogeneous boundary conditions in problems \eqref{eq:main-equation} and \eqref{eq:linearized-equation}, we have that 
$$\nabla v=\frac{\partial v}{\partial\nu}\nu=\frac{\partial v}{\partial\nu}x \quad \hbox{and} \quad \nabla \psi=\frac{\partial \psi}{\partial\nu}\nu=\frac{\partial \psi}{\partial\nu}x \quad \hbox{on} \quad \partial B_{1},$$ which implies
\begin{align*}
&\quad -\int_{\partial B_1}
\frac{\partial v}{\partial \nu}\,X\cdot\nabla\psi\,d\sigma
-\int_{\partial B_1}
\frac{\partial \psi}{\partial \nu}\,X\cdot\nabla v\,d\sigma
+\int_{\partial B_1}
(\nabla v\cdot\nabla\psi)\,X\cdot\nu\,d\sigma  \\
&=
-\int_{\partial B_1}
\frac{\partial v}{\partial \nu}
\frac{\partial \psi}{\partial \nu}
X\cdot\nu\,d\sigma
-\int_{\partial B_1}
\frac{\partial \psi}{\partial \nu}
\frac{\partial v}{\partial \nu}
X\cdot\nu\,d\sigma
+\int_{\partial B_1}
\frac{\partial v}{\partial \nu}
\frac{\partial \psi}{\partial \nu}
X\cdot\nu\,d\sigma  \\
&=
-\int_{\partial B_1}
\frac{\partial v}{\partial \nu}
\frac{\partial \psi}{\partial \nu}
X\cdot\nu\,d\sigma  \\
&=
-\int_{\partial B_1}
\frac{\partial v}{\partial \nu}
\frac{\partial \psi}{\partial \nu}
(x_1^2-x_2^2)\,d\sigma .
\end{align*}
On the other hand, since $\nu=-\frac{x}{\varepsilon}$ on $\partial B_{\varepsilon}$,  we deduce
\begin{align*}
&\quad-\int_{\partial B_\varepsilon}
\frac{\partial v}{\partial \nu}X\cdot\nabla\psi\,d\sigma
-\int_{\partial B_\varepsilon}
\frac{\partial \psi}{\partial \nu}X\cdot\nabla v\,d\sigma
+\int_{\partial B_\varepsilon}
(\nabla v\cdot\nabla\psi)X\cdot\nu\,d\sigma  \\
&=
\frac{1}{\varepsilon^3}
\int_{\partial B_\varepsilon}
\bigg[
\left(
\frac{\partial v}{\partial x_1}x_1
+
\frac{\partial v}{\partial x_2}x_2
\right)
\left(
\frac{\partial \psi}{\partial x_1}x_1
-
\frac{\partial \psi}{\partial x_2}x_2
\right)  \\
&\qquad\qquad
+
\left(
\frac{\partial \psi}{\partial x_1}x_1
+
\frac{\partial \psi}{\partial x_2}x_2
\right)
\left(
\frac{\partial v}{\partial x_1}x_1
-
\frac{\partial v}{\partial x_2}x_2
\right)  \\
&\qquad\qquad
-
\left(
\frac{\partial v}{\partial x_1}
\frac{\partial \psi}{\partial x_1}
+
\frac{\partial v}{\partial x_2}
\frac{\partial \psi}{\partial x_2}
\right)
(x_1^2-x_2^2)
\bigg]\,d\sigma  \\
&=
\frac{1}{\varepsilon}
\int_{\partial B_\varepsilon}
\left(
\frac{\partial v}{\partial x_1}
\frac{\partial \psi}{\partial x_1}
-
\frac{\partial v}{\partial x_2}
\frac{\partial \psi}{\partial x_2}
\right)\,d\sigma  \\
&=
2\pi
\frac{\partial v}{\partial x_1}(0)
\frac{\partial \psi}{\partial x_1}(0)
-
2\pi
\frac{\partial v}{\partial x_2}(0)
\frac{\partial \psi}{\partial x_2}(0)
+o(1),
\end{align*}
as $\varepsilon\rightarrow0^{+}$. 

By inserting the above two boundary estimates into \eqref{simplify int Delta v + Delta psi}, we derive
\begin{equation}\label{simplify 2 int Delta v + Delta psi}
    \begin{aligned}
&\quad\int_{\Omega_\varepsilon}
\left[
(-\Delta v)(X\cdot\nabla\psi)
+
(-\Delta\psi)(X\cdot\nabla v)
\right]\,dx \\
&=
-\int_{\partial B_1}
\frac{\partial v}{\partial \nu}
\frac{\partial \psi}{\partial \nu}
(x_1^2-x_2^2)\,d\sigma
+
2\pi
\frac{\partial v}{\partial x_1}(0)
\frac{\partial \psi}{\partial x_1}(0)
-
2\pi
\frac{\partial v}{\partial x_2}(0)
\frac{\partial \psi}{\partial x_2}(0)
+o(1),
\end{aligned}
\end{equation}
as $\varepsilon\rightarrow0^{+}$.

We are now in a position to examine the right hand side of \eqref{eq:two equation time X dot nabla and sum}: Integrating by parts we obtain
\begin{equation}\label{eq:right simplify}
\begin{aligned}
&\quad\quad \lambda\int_{B_1} V|x|^2e^v
\left[X\cdot\nabla\psi+\psi X\cdot\nabla v\right]\,dx  \\
&\quad
= \lambda\int_{B_1} V|x|^2 X\cdot\nabla(e^v\psi)\,dx  \\
&\quad
= \lambda\int_{\partial B_1} V|x|^2e^v\psi\,X\cdot\nu\,d\sigma
-\lambda\int_{B_1} e^v\psi\,\operatorname{div}\!\left(V|x|^2X\right)\,dx  \\
&\quad
= -\lambda\int_{B_1} e^v\psi
\left(
\frac{\partial V}{\partial x_1}x_1
-
\frac{\partial V}{\partial x_2}x_2
\right)\,dx,
\end{aligned}
\end{equation}
where we have used the homogeneous boundary condition $\psi=0$ on $\partial B_{1}$. Substituting \eqref{simplify 2 int Delta v + Delta psi} and \eqref{eq:right simplify} into \eqref{eq:two equation time X dot nabla and sum} and letting $\varepsilon\rightarrow0$, we arrive at \eqref{eq:Pohozaev one}.

Similarly, we define
\begin{equation}\label{eq:Y(x)}
    Y(x):=
\left(
\frac{x_2}{|x|^2},
\frac{x_1}{|x|^2}
\right),
\quad x\neq0.
\end{equation}
Multiplying \eqref{eq:main-equation} by $Y\cdot \nabla\psi$ and \eqref{eq:linearized-equation}
by $Y\cdot \nabla v$ and integrating on $\Omega_{\varepsilon}$, we deduce
\begin{equation}\label{eq:two equation time Y dot nabla and sum}
\int_{\Omega_\varepsilon}
\left[
(-\Delta v)(Y\cdot\nabla\psi)
+
(-\Delta\psi)(Y\cdot\nabla v)
\right]\,dx
=
\lambda\int_{\Omega_\varepsilon}
V|x|^2e^v
\left[
Y\cdot\nabla\psi
+
\psi Y\cdot\nabla v
\right]\,dx .
\end{equation}
After integrating the left hand side by parts, we obtain
\begin{equation}\label{eq:int Omega e for Delta v Y}
\int_{\Omega_\varepsilon}
(-\Delta v)(Y\cdot\nabla\psi)\,dx
=
-\int_{\partial\Omega_\varepsilon}
\frac{\partial v}{\partial\nu}\, Y\cdot\nabla\psi\,d\sigma
+
\int_{\Omega_\varepsilon}
\nabla v\cdot\nabla(Y\cdot\nabla\psi)\,dx,
\end{equation}
and
\begin{equation}\label{eq:int Omega e for Delta psi Y}
\int_{\Omega_\varepsilon}
(-\Delta\psi)(Y\cdot\nabla v)\,dx
=
-\int_{\partial\Omega_\varepsilon}
\frac{\partial\psi}{\partial\nu}\, Y\cdot\nabla v\,d\sigma
+
\int_{\Omega_\varepsilon}
\nabla\psi\cdot\nabla(Y\cdot\nabla v)\,dx .
\end{equation}
Expanding the interior term leads to
\begin{equation*}
    \begin{aligned}
\nabla v\cdot\nabla(Y\cdot\nabla\psi)
=
\frac{\partial v}{\partial x_1}
\frac{\partial}{\partial x_1}
\left(
Y_1\frac{\partial\psi}{\partial x_1}
+
Y_2\frac{\partial\psi}{\partial x_2}
\right)
+
\frac{\partial v}{\partial x_2}
\frac{\partial}{\partial x_2}
\left(
Y_1\frac{\partial\psi}{\partial x_1}
+
Y_2\frac{\partial\psi}{\partial x_2}
\right)
\end{aligned}
\end{equation*}
and
\[
\begin{aligned}
\nabla\psi\cdot\nabla(Y\cdot\nabla v)
=
\frac{\partial \psi}{\partial x_1}
\frac{\partial}{\partial x_1}
\left(
Y_1\frac{\partial v}{\partial x_1}
+
Y_2\frac{\partial v}{\partial x_2}
\right)
+
\frac{\partial \psi}{\partial x_2}
\frac{\partial}{\partial x_2}
\left(
Y_1\frac{\partial v}{\partial x_1}
+
Y_2\frac{\partial v}{\partial x_2}
\right).
\end{aligned}
\]
Arguing as in the derivation of \eqref{eq:expand nabla} and \eqref{eq:X dot nabla nabla v dot nabla psi}, we get
\begin{equation}\label{eq:Y dot nabla nabla v dot nabla psi}
\begin{aligned}
Y_1\frac{\partial}{\partial x_1}
\left(
\frac{\partial v}{\partial x_1}
\frac{\partial\psi}{\partial x_1}
+
\frac{\partial v}{\partial x_2}
\frac{\partial\psi}{\partial x_2}
\right)
+
Y_2\frac{\partial}{\partial x_2}
\left(
\frac{\partial v}{\partial x_1}
\frac{\partial\psi}{\partial x_1}
+
\frac{\partial v}{\partial x_2}
\frac{\partial\psi}{\partial x_2}
\right) =
Y\cdot\nabla(\nabla v\cdot\nabla\psi).
\end{aligned}
\end{equation}
Integrating \eqref{eq:Y dot nabla nabla v dot nabla psi} by parts on $\Omega_{\varepsilon}$, we obtain
\begin{equation}\label{eq:int Omega e for Y dot nabla nabla v dot nabla psi}
\int_{\Omega_\varepsilon}
Y\cdot\nabla(\nabla v\cdot\nabla\psi)\,dx
=
\int_{\partial\Omega_\varepsilon}
(\nabla v\cdot\nabla\psi)Y\cdot\nu\,d\sigma
-\int_{\Omega_\varepsilon}
(\nabla v\cdot\nabla\psi)\operatorname{div}Y\,dx .
\end{equation}
In view of \eqref{eq:int Omega e for Delta v Y}, \eqref{eq:int Omega e for Delta psi Y} and \eqref{eq:int Omega e for Y dot nabla nabla v dot nabla psi}, we derive
\begin{equation}\label{int Delta v + Delta psi Y}
\begin{aligned}
&\quad \int_{\Omega_\varepsilon}
\left[
(-\Delta v)(Y\cdot\nabla\psi)
+
(-\Delta\psi)(Y\cdot\nabla v)
\right]\,dx \\
&=
-\int_{\partial\Omega_\varepsilon}
\frac{\partial v}{\partial\nu}\, Y\cdot\nabla\psi\,d\sigma
-\int_{\partial\Omega_\varepsilon}
\frac{\partial\psi}{\partial\nu}\, Y\cdot\nabla v\,d\sigma \\
&\quad
+
\int_{\partial\Omega_\varepsilon}
(\nabla v\cdot\nabla\psi)Y\cdot\nu\,d\sigma 
+
\int_{\Omega_\varepsilon}
\bigg[
2\frac{\partial Y_1}{\partial x_1}
\frac{\partial v}{\partial x_1}
\frac{\partial\psi}{\partial x_1} \\
&\qquad
+
\left(
\frac{\partial Y_2}{\partial x_1}
+
\frac{\partial Y_1}{\partial x_2}
\right)
\left(
\frac{\partial v}{\partial x_1}
\frac{\partial\psi}{\partial x_2}
+
\frac{\partial v}{\partial x_2}
\frac{\partial\psi}{\partial x_1}
\right) 
+
2\frac{\partial Y_2}{\partial x_2}
\frac{\partial v}{\partial x_2}
\frac{\partial\psi}{\partial x_2} \\
&\qquad
-
\left(
\frac{\partial Y_1}{\partial x_1}
+
\frac{\partial Y_2}{\partial x_2}
\right)
\left(
\frac{\partial v}{\partial x_1}
\frac{\partial\psi}{\partial x_1}
+
\frac{\partial v}{\partial x_2}
\frac{\partial\psi}{\partial x_2}
\right)
\bigg]\,dx .
\end{aligned}
\end{equation}

Recall that $Y(x)$ defined in \eqref{eq:Y(x)}, it holds that
\begin{equation}\label{partian Y1 and Y2}
\frac{\partial Y_1}{\partial x_1}
=
\frac{\partial Y_2}{\partial x_2}
=
-\frac{2x_{1}x_{2}}{|x|^4},
\quad
\frac{\partial Y_2}{\partial x_1}
=
-\frac{\partial Y_1}{\partial x_2}
=
\frac{x_{2}^{2}-x_{1}^{2}}{|x|^4}.
\end{equation}
Substituting \eqref{partian Y1 and Y2} into \eqref{int Delta v + Delta psi Y}, we have
\begin{equation}\label{simplify int Delta v + Delta psi Y}
\begin{aligned}
&\quad \int_{\Omega_\varepsilon}
\left[
(-\Delta v)(Y\cdot\nabla\psi)
+
(-\Delta\psi)(Y\cdot\nabla v)
\right]\,dx  \\
&=
-\int_{\partial\Omega_\varepsilon}
\frac{\partial v}{\partial\nu}\,Y\cdot\nabla\psi\,d\sigma
-\int_{\partial\Omega_\varepsilon}
\frac{\partial\psi}{\partial\nu}\,Y\cdot\nabla v\,d\sigma
+\int_{\partial\Omega_\varepsilon}
(\nabla v\cdot\nabla\psi)\,Y\cdot\nu\,d\sigma.
\end{aligned}
\end{equation}
Now let us focus on the boundary terms: the homogeneous boundary conditions in \eqref{eq:main-equation} and \eqref{eq:linearized-equation} yields 
$$\nabla v=\frac{\partial v}{\partial\nu}\nu=\frac{\partial v}{\partial\nu}x \quad \hbox{and} \quad \nabla \psi=\frac{\partial \psi}{\partial\nu}\nu=\frac{\partial \psi}{\partial\nu}x\quad \hbox{on}\quad \partial B_{1},$$ which implies
\begin{equation*}
\begin{aligned}
&\quad -\int_{\partial B_1}
\frac{\partial v}{\partial \nu}\,Y\cdot\nabla\psi\,d\sigma
-\int_{\partial B_1}
\frac{\partial \psi}{\partial \nu}\,Y\cdot\nabla v\,d\sigma
+\int_{\partial B_1}
(\nabla v\cdot\nabla\psi)\,Y\cdot\nu\,d\sigma  \\
&=
-\int_{\partial B_1}
\frac{\partial v}{\partial \nu}
\frac{\partial \psi}{\partial \nu}
Y\cdot\nu\,d\sigma
-\int_{\partial B_1}
\frac{\partial \psi}{\partial \nu}
\frac{\partial v}{\partial \nu}
Y\cdot\nu\,d\sigma
+\int_{\partial B_1}
\frac{\partial v}{\partial \nu}
\frac{\partial \psi}{\partial \nu}
Y\cdot\nu\,d\sigma  \\
&=
-\int_{\partial B_1}
\frac{\partial v}{\partial \nu}
\frac{\partial \psi}{\partial \nu}
Y\cdot\nu\,d\sigma  \\
&=
-2\int_{\partial B_1}
\frac{\partial v}{\partial \nu}
\frac{\partial \psi}{\partial \nu}
x_{1}x_{2}\,d\sigma .
\end{aligned}
\end{equation*}
Since $\nu=-\frac{x}{\varepsilon}$ on $\partial B_{\varepsilon}$, a direct computation gives
\[
\begin{aligned}
&\quad -\int_{\partial B_\varepsilon}
\frac{\partial v}{\partial \nu}Y\cdot\nabla\psi\,d\sigma
-\int_{\partial B_\varepsilon}
\frac{\partial \psi}{\partial \nu}Y\cdot\nabla v\,d\sigma
+\int_{\partial B_\varepsilon}
(\nabla v\cdot\nabla\psi)Y\cdot\nu\,d\sigma  \\
&=
\frac{1}{\varepsilon^3}
\int_{\partial B_\varepsilon}
\bigg[
\left(
\frac{\partial v}{\partial x_1}x_1
+
\frac{\partial v}{\partial x_2}x_2
\right)
\left(
\frac{\partial \psi}{\partial x_1}x_2
+
\frac{\partial \psi}{\partial x_2}x_1
\right)  \\
&\quad
+
\left(
\frac{\partial \psi}{\partial x_1}x_1
+
\frac{\partial \psi}{\partial x_2}x_2
\right)
\left(
\frac{\partial v}{\partial x_1}x_2
+
\frac{\partial v}{\partial x_2}x_1
\right) 
-
\left(
\frac{\partial v}{\partial x_1}
\frac{\partial \psi}{\partial x_1}
+
\frac{\partial v}{\partial x_2}
\frac{\partial \psi}{\partial x_2}
\right)
2x_{1}x_{2}
\bigg]\,d\sigma  \\
&=
\frac{1}{\varepsilon}
\int_{\partial B_\varepsilon}
\left(
\frac{\partial v}{\partial x_1}
\frac{\partial \psi}{\partial x_2}
+
\frac{\partial v}{\partial x_2}
\frac{\partial \psi}{\partial x_1}
\right)\,d\sigma  \\
&=
2\pi
\frac{\partial v}{\partial x_1}(0)
\frac{\partial \psi}{\partial x_2}(0)
+
2\pi
\frac{\partial v}{\partial x_2}(0)
\frac{\partial \psi}{\partial x_1}(0)
+o(1),
\end{aligned}
\]
as $\varepsilon\rightarrow0^{+}$. 

By inserting the above two boundary estimates into \eqref{simplify int Delta v + Delta psi Y}, we derive
\begin{equation}\label{simplify 2 int Delta v + Delta psi Y}
    \begin{aligned}
&\quad\int_{\Omega_\varepsilon}
\left[
(-\Delta v)(Y\cdot\nabla\psi)
+
(-\Delta\psi)(Y\cdot\nabla v)
\right]\,dx \\
&=
-2\int_{\partial B_1}
\frac{\partial v}{\partial \nu}
\frac{\partial \psi}{\partial \nu}
x_{1}x_{2}\,d\sigma
+
2\pi
\frac{\partial v}{\partial x_1}(0)
\frac{\partial \psi}{\partial x_2}(0)
+
2\pi
\frac{\partial v}{\partial x_2}(0)
\frac{\partial \psi}{\partial x_1}(0)
+o(1),
\end{aligned}
\end{equation}
as $\varepsilon\rightarrow0^{+}$.

We proceed to estimate the right hand side of \eqref{eq:two equation time Y dot nabla and sum} as in \eqref{eq:right simplify} and obtain 
\begin{equation}\label{eq:right simplify Y}
\begin{aligned}
&\quad\quad \lambda\int_{B_1} V|x|^2e^v
\left[Y\cdot\nabla\psi+\psi Y\cdot\nabla v\right]\,dx  \\
&\quad
= \lambda\int_{B_1} V|x|^2 Y\cdot\nabla(e^v\psi)\,dx  \\
&\quad
= \lambda\int_{\partial B_1} V|x|^2e^v\psi\,Y\cdot\nu\,d\sigma
-\lambda\int_{B_1} e^v\psi\,\operatorname{div}\!\left(V|x|^2Y\right)\,dx  \\
&\quad
= -\lambda\int_{B_1} e^v\psi
\left(
\frac{\partial V}{\partial x_1}x_2
+
\frac{\partial V}{\partial x_2}x_1
\right)\,dx,
\end{aligned}
\end{equation}
where we have used the homogeneous boundary condition $\psi=0$ on $\partial B_{1}$. Substituting \eqref{simplify 2 int Delta v + Delta psi Y} and \eqref{eq:right simplify Y} into \eqref{eq:two equation time Y dot nabla and sum} and letting $\varepsilon\rightarrow0$, we conclude \eqref{eq:Pohozaev two}.
\end{proof}

We now prove Theorem \ref{thm nondegenerate}, arguing by contradiction. Suppose that there are $m\rightarrow+\infty$, satisfying $\|\psi_{m}\|=1$, and
\begin{equation*}
    \widehat{L}_{{m}}\psi_{m}=0.
\end{equation*}
Let
\begin{equation}\label{phi tuta m(y) and rho m}
    \widetilde{\psi}_{m}(y)=\psi_{m}(\rho_{m}y),
\end{equation}
where $\rho_m$ is defined in Theorem A.
\begin{lemma}\label{lem:the convergence of psi m tuta}
It holds
\begin{equation*}
\widetilde\psi_m
\longrightarrow
a_0Z_0+a_1Z_1+a_2Z_2,
\end{equation*}
uniformly in $C^1(B_R(0))$ for any $R>0$, where $a_{0}$, $a_1$ and $a_2$ are some constants and $Z_0$, $Z_1$, $Z_2$ are functions defined in \eqref{eq:Z012} with $c=\xi\in\mathbb{R}^{2}$.
\end{lemma}

\begin{proof}
We first consider the scaling of $v_m$. Define
\[
\widetilde v_m(y)
:=
v_m(\rho_m y)+2\log\lambda_m-\log 32.
\]
We claim that
\[
\widetilde v_m
\to
U_{\xi}
\quad\text{in }C^1_{\mathrm{loc}}(\mathbb R^2),
\]
where
\[
U_{\xi}(y)
=
\log
\frac{32}
{\left(1+|y^2-\xi|^2\right)^2}.
\]

Indeed, in view of \eqref{eq:vm=PWm+phi m}, we have
\[
\widetilde v_m(y)
=
P W_m(\rho_m y)+2\log\lambda_m-\log 32
+
\phi_m(\rho_m y).
\]
Next we show that
\[
\phi_m(\rho_m y)\to0
\quad\text{in }C^1_{\mathrm{loc}}(\mathbb R^2).
\]
Set
\[
\widetilde\phi_m(y):=\phi_m(\rho_m y).
\]
For every fixed $R>0$, by the change of variables $x=\rho_m y$, we derive
\[
\int_{B_R}|\nabla_y\widetilde\phi_m|^2\,dy
=
\int_{B_{\rho_m R}}|\nabla_x\phi_m|^2\,dx
\le
\|\phi_m\|^2.
\]
It follows from \eqref{eq:phim H01 estimate} that
\begin{equation}\label{eq:nabla 2 norm to 0}
\|\nabla\widetilde\phi_m\|_{L^2(B_R)}
\le
C\lambda_m^{\frac12-\varepsilon}
\to0.
\end{equation}
A direct computation yields
\[
\|\widetilde\phi_m\|_{L^2(B_R)}
=
\rho_m^{-1}
\|\phi_m\|_{L^2(B_{\rho_m R})}\le
C\rho_m^{-1}\|\phi_m\|_{H_0^1(B_1)},
\]
which implies
\begin{equation}\label{eq:2 norm to 0}
\|\widetilde\phi_m\|_{L^2(B_R)}
\le
C\lambda_m^{-1/4}\lambda_m^{\frac12-\varepsilon}
=
C\lambda_m^{\frac14-\varepsilon}
\to0,
\end{equation}
for $\varepsilon<1/4$. A combination of \eqref{eq:nabla 2 norm to 0} and \eqref{eq:2 norm to 0} leads to
\[
\widetilde\phi_m\to0
\quad\text{strongly in }H^1(B_R).
\]

Next, since $v_m=P W_m+\phi_m$ is a solution of \eqref{eq:main-equation} and $$-\Delta P W_m=|x|^2e^{W_m},$$ while the correction term $\phi_m$ satisfies
\begin{equation}\label{eq:correction term phi m equation}
    -\Delta\phi_m
=
\lambda_m V(x)|x|^2e^{P W_m}
\left(e^{\phi_m}-1\right)
+
\left[
\lambda_m V(x)|x|^2e^{P W_m}
-
|x|^2e^{W_m}
\right].
\end{equation}
After the scaling $x=\rho_m y$ and using the definition of $\rho_m$,
we obtain
\begin{equation}\label{eq:Delta tilde phi m}
-\Delta_y\widetilde\phi_m
=
\mathcal{A}_m(y)\left(e^{\widetilde\phi_m}-1\right)
+
D_m(y),
\end{equation}
where
\[
\mathcal{A}_m(y)
=
\lambda_m\rho_m^4 |y|^2
V(\rho_m y)e^{P W_m(\rho_m y)}
=
|y|^2V(\rho_m y)
e^{P W_m(\rho_m y)+2\log\lambda_m-\log32},
\]
and
\[
D_m(y)
=
|y|^2
\left[
V(\rho_m y)
e^{P W_m(\rho_m y)+2\log\lambda_m-\log32}
-
e^{W_m(\rho_m y)+\log\lambda_m-\log32}
\right].
\]
By \eqref{eq:Pw expand}, we have
\begin{equation*}
    P W_m(\rho_m y)+2\log\lambda_m-\log32
=
W_m(\rho_m y)+\log\lambda_m-\log32
+8\pi \mathcal H(\rho_{m}^{2}y^2,b_m)
+O(\lambda_m)
\end{equation*}
locally uniformly in $y$. 
 
 Since $V(\rho_m y)\to V(0)=1$ and $8\pi \mathcal H(\rho_{m}^{2}y^2,b_m)\to 8\pi \mathcal H(0,0)=0$ locally uniformly, it follows that $A_m$ is locally uniformly bounded and $D_m\to0$ locally uniformly. Since
\[
\widetilde\phi_m\to0
\quad\text{in }H^1(B_R),
\]
the Sobolev embedding in dimension two gives
\[
\widetilde\phi_m\to0
\quad\text{in }L^q(B_R)
\quad\text{for }1<q<+\infty.
\]
It follows from Lemma \ref{lem:M-T inequality} that
\[
e^{\widetilde\phi_m}-1\to0
\quad\text{in }L^q(B_R)
\quad\text{for }1<q<+\infty.
\]
By \eqref{eq:Delta tilde phi m}, we deduce that
\[
-\Delta\widetilde\phi_m\to0
\quad\text{in }L^q(B_R)
\quad\text{for }1<q<+\infty.
\]
Then interior elliptic estimate \cite[Theorem 9.11]{G-T2001} yields
\[
\|\widetilde\phi_m\|_{W^{2,q}(B_{R/2})}
\le
C
\left(
\|\widetilde\phi_m\|_{L^q(B_R)}
+
\|-\Delta\widetilde\phi_m\|_{L^q(B_R)}
\right)
\to0.
\]
Choosing $q>2$, the Morrey embedding gives
\[
W^{2,q}(B_{R/2}) \hookrightarrow C^{1,\alpha}(B_{R/2}),
\]
where $\alpha=1-\frac{2}{q}$. Hence
\[
\widetilde{\phi}_m \to 0
\quad \text{in } C^1(B_{R/2}).
\]
Since \(R>0\) is arbitrary, we conclude that
\begin{equation}\label{eq:phi m tuta to 0}
    \widetilde{\phi}_m \to 0
\quad \text{in } C^1_{\mathrm{loc}}(\mathbb R^2).
\end{equation}

On the other hand, it follows from \eqref{eq:bubble} that
\[
W_m(\rho_m y)+\log\lambda_m-\log32
=
\log
\frac{32}
{\left(
1+\left|y^2-\frac{b_m}{\rho_m^2}\right|^2
\right)^2}.
\]
By Theorem A, 
\[
\xi_m=\frac{b_m}{\rho_m^2}\to \xi\in\mathbb{R}^{2},
\]
we obtain
\[
W_m(\rho_m y)+\log\lambda_m-\log32
\to
\log
\frac{32}
{\left(1+|y^2-\xi|^2\right)^2}
=
U_{\xi}(y) \quad \hbox{ in}\quad  C^1_{\mathrm{loc}}(\mathbb R^2).
\]
 Together with $
\widetilde\phi_m\to0$ and $8\pi \mathcal H(\rho_{m}^{2}y^2,b_m)\to 0$ in $C^1_{\mathrm{loc}}(\mathbb R^2)$, we obtain
\begin{equation}\label{eq:convergence of tuta vm}
\widetilde v_m
\to
U_{\xi}
\quad\text{in }C^1_{\mathrm{loc}}(\mathbb R^2).
\end{equation}
In particular,
$U_{\xi}(y)$ satisfies $$
-\Delta U_{\xi}
=
|y|^2e^{U_{\xi}}
\quad\text{in }\mathbb R^2.
$$

In the following, we rescale the linearized equation \eqref{eq:linearized-equation}. Recall that $\widetilde\psi_m(y)$ defined in \eqref{phi tuta m(y) and rho m}, a direct computation gives
\[
-\Delta_y\widetilde\psi_m
=
\rho_m^2\lambda_m
V(\rho_m y)|\rho_m y|^2
e^{v_m(\rho_m y)}
\widetilde\psi_m.
\]
We rewrite the above equation as
\begin{equation}\label{eq:rewritten equation for psi m tuta}
-\Delta\widetilde\psi_m
=
\Phi_m(y)\widetilde\psi_m,
\end{equation}
where
\[
\Phi_m(y)
=
|y|^2V(\rho_m y)
e^{v_m(\rho_m y)+2\log\lambda_m-\log32}.
\]
Using \eqref{eq:convergence of tuta vm}, we obtain
\begin{equation}\label{eq:convergence of Qm}
\Phi_m\to \Phi_{\xi}
\quad\text{locally uniformly in }\mathbb R^2,
\end{equation}
where
\[
\Phi_{\xi}(y)
=
|y|^2e^{U_{\xi}(y)}
=
\frac{32|y|^2}
{\left(1+|y^2-\xi|^2\right)^2}.
\]

We now prove local compactness of $\widetilde\psi_m$. For every fixed $R>0$, we have
\[
\int_{B_R}|\nabla\widetilde\psi_m|^2\,dy
=
\int_{B_{\rho_m R}}|\nabla\psi_m|^2\,dx
\le
\int_{B_1}|\nabla\psi_m|^2\,dx
=
1.
\]
Then  multiply \eqref{eq:linearized-equation} by $\psi_m$ and integrate over $B_1$ to derive
\begin{equation*}
\int_{B_1}|\nabla\psi_m|^2\,dx
=
\int_{B_1}
\lambda_m V(x)|x|^2e^{v_m}\psi_m^2\,dx,
\end{equation*}
which means
\[
\int_{B_1}
\lambda_m V(x)|x|^2e^{v_m}\psi_m^2\,dx
=
1.
\]
Taking the scaling $x=\rho_m y$, we obtain
\[
\int_{B_{\rho_m^{-1}}}
\Phi_m(y)\widetilde\psi_m^2(y)\,dy
=
1.
\]
Fix $R>0$ and let
\[
A_R
=
\left\{
y\in\mathbb R^2:
\frac R2<|y|<R
\right\}.
\]
Since $\Phi_{\xi}(y)>0$ and $\Phi_m\to \Phi_{\xi}$ locally uniformly on $A_{R}$, there exists $c_R>0$ such that
\[
\Phi_m(y)\ge c_R
\quad\text{for }y\in A_R
\]
for sufficiently large $m$. Therefore
\[
\int_{A_R}\widetilde\psi_m^2\,dy
\le
\frac1{c_R}
\int_{A_R}\Phi_m\widetilde\psi_m^2\,dy
\le
\frac1{c_R}.
\]
We {\bf{claim }}the elementary inequality holds:
\begin{equation}\label{eq: a elementary ineq}
\|u\|_{L^2(B_R)}
\le
C_R
\left(
\|\nabla u\|_{L^2(B_R)}
+
\|u\|_{L^2(A_R)}
\right),
\quad u\in H^1(B_R).
\end{equation}
Indeed, if this inequality failed, one could find $u_k\in H^1(B_R)$ such that
\begin{equation*}
    \|u_k\|_{L^2(B_R)}=1,
\quad
\|\nabla u_k\|_{L^2(B_R)}\to0,
\quad
\|u_k\|_{L^2(A_R)}\to0.
\end{equation*}
Then $u_k$ would converge strongly in $L^2(B_R)$ to a constant, and the vanishing on
$A_R$ would force this constant to be zero, contradicting
$\|u_k\|_{L^2(B_R)}=1$.

Replacing $u$ by $\widetilde\psi_m$ in \eqref{eq: a elementary ineq}, we get
\[
\|\widetilde\psi_m\|_{L^2(B_R)}
\le C_R,
\]
which implies
\[
\|\widetilde\psi_m\|_{H^1(B_R)}
\le C_R.
\]
Consequently, up to a subsequence, one obtains
\[
\widetilde\psi_m\rightharpoonup \psi_0
\quad\text{weakly in }H^1(B_R),
\]
and
\begin{equation}\label{eq:convergence of psi m tuta}
\widetilde\psi_m\to\psi_0
\quad\text{strongly in }L^q(B_R)\quad\text{for } 1<q<+\infty.
\end{equation}

By \eqref{eq:rewritten equation for psi m tuta}, for any $\varphi\in C_c^\infty(\mathbb R^2)$, we obtain
\[
\int_{\mathbb R^2}
\nabla\widetilde\psi_m\cdot\nabla\varphi\,dy
=
\int_{\mathbb R^2}
\Phi_m\widetilde\psi_m\varphi\,dy,
\]
together with \eqref{eq:convergence of Qm} and \eqref{eq:convergence of psi m tuta}, we pass to the limit and obtain
\[
\int_{\mathbb R^2}
\nabla\psi_0\cdot\nabla\varphi\,dy
=
\int_{\mathbb R^2}
\Phi_{\xi}\psi_0\varphi\,dy.
\]
Therefore
\[
-\Delta\psi_0
=
\Phi_{\xi}\psi_0
=
\frac{32|y|^2}
{\left(1+|y^2-\xi|^2\right)^2}
\psi_0
\quad\text{in }\mathbb R^2.
\]
Moreover, by lower semicontinuity,
\[
\int_{B_R}|\nabla\psi_0|^2\,dy
\le
\liminf_{m\to\infty}
\int_{B_R}|\nabla\widetilde\psi_m|^2\,dy
\le 1.
\]
Letting $R\to+\infty$, we have
\[
\int_{\mathbb R^2}|\nabla\psi_0|^2\,dy<+\infty.
\]

Finally, by Lemma \ref{lem:nondegeneracy of finite mass solutions of the singular Liouville equation}, we deduce that there exist constants $a_0,a_1,a_2\in\mathbb R$ such that
\[
\psi_0=a_0Z_0+a_1Z_1+a_2Z_2,
\]
where $Z_0,\,Z_{1},\,Z_{2}$ are defined in \eqref{eq:Z012}. Therefore, by \eqref{eq:convergence of psi m tuta} with $q>2$, interior elliptic estimates and Morrey embedding,
\[
\widetilde\psi_m
\longrightarrow
a_0Z_0+a_1Z_1+a_2Z_2,
\]
uniformly in $C^1(B_R(0))$ for any $R>0$. The proof is complete.
\end{proof}

Let
\[
K_{m}=\operatorname{span}\{PZ_{m}^{1},\,PZ_{m}^{2}\}.
\]
We decompose
\begin{equation}\label{decompose psi m}
    \psi_{m}=c_{1,m}PZ_{m}^{1}+c_{2,m}PZ_{m}^{2}+\psi_{m}^{\ast},
\end{equation}
where $\psi_{m}^{\ast}\in K_{m}^{\perp}$, namely
\[
\int_{B_1}\nabla\psi_m^*\nabla PZ_m^j\,dx=0,
\quad j=1,2.
\]
It follows from Lemma \ref{lem:the convergence of psi m tuta} that $c_{1,m}$ and $c_{2,m}$ are bounded.
\begin{lemma}\label{lem:estimate for psi m star}
There exists $\eta>0$ such that
\begin{equation}\label{eq:psi m star H01 estimate}
\|\psi_{m}^{\ast}\|\leq C\lambda_{m}^{\frac{1}{4}+\eta}(|c_{1,m}|+|c_{2,m}|)= o(1)(|c_{1,m}|+|c_{2,m}|).
\end{equation}
\end{lemma}
\begin{proof}
We first recall a preliminary argument from \cite{D-W-Z2026}.
    For any $p>1$, let
\[
i_p^* : L^p(B_1) \longrightarrow H_0^1(B_1)
\]
be the adjoint operator of the embedding $$
i_p : H_0^1(B_1) \hookrightarrow L^{\frac{p}{p-1}}(B_1),$$ 
we have $u=i_p^*(v)$ if and only if
$-\Delta u=v$ in $B_1$, $u=0$  on $\partial B_1.$ We point out that \(i_p^*\) is a continuous mapping, namely
\begin{equation}\label{eq:boundedness of ipstar}
    \|i_p^*(v)\| \leq c_p \|v\|_p,
\quad \text{for any } v\in L^p(B_1),
\end{equation}
for some constant \(c_p\) depending only on \(p\).

In terms of $i_{p}^{\ast}$, equation \eqref{eq:main-equation} reduces to
\begin{equation*}
    v=i_p^*(\lambda V|x|^{2}e^{v}).
\end{equation*}
Define
\[
S_\lambda(v):=v-i_p^*\bigl(\lambda V|x|^2e^v\bigr).
\]
Then the linearized operator is given by
\[
DS_\lambda(v_m)[\psi]
=
\psi-i_p^*\bigl(\lambda_m V|x|^2e^{v_m}\psi\bigr).
\]
Therefore,
\[
-\Delta \psi_m=\lambda_m V|x|^2e^{v_m}\psi_m
\]
is equivalent to
\[
DS_{\lambda_m}(v_m)[\psi_m]=0.
\]
Substituting \eqref{decompose psi m} into the above identity, we obtain
\[
DS_{\lambda_m}(v_m)
\bigl[
c_{1,m}PZ_m^1+c_{2,m}PZ_m^2+\psi_m^*
\bigr]=0.
\]
Since \(DS_{\lambda_m}(v_m)\) is linear, it follows that
\[
DS_{\lambda_m}(v_m)[\psi_m^*]
+
\sum_{j=1}^2 c_{j,m}
DS_{\lambda_m}(v_m)[PZ_m^j]
=0.
\]
Let $\Pi_m^\perp:H_0^1(B_1)\to K_m^\perp$ denote the projection onto \(K_m^\perp\). Then we get
\begin{equation}\label{Pi m DS}
\Pi_m^\perp DS_{\lambda_m}(v_m)[\psi_m^*]
=
-\sum_{j=1}^2 c_{j,m}\,
\Pi_m^\perp DS_{\lambda_m}(v_m)[PZ_m^j].
\end{equation}
Define
\begin{equation}\label{eq:Lm tuta}
    \widetilde{L}_m
:=
\Pi_m^\perp DS_{\lambda_m}(v_m)\big|_{K_m^\perp}.
\end{equation}
Since \(\psi_m^*\in K_m^\perp\), \eqref{Pi m DS} becomes
\begin{equation}\label{eq:Lm tuta psi m star}
\widetilde{L}_m \psi_m^*
=
-\sum_{j=1}^2 c_{j,m}R_{j,m},
\end{equation}
where
\begin{equation}\label{eq:definition of Rjm}
R_{j,m}
=
\Pi_m^\perp DS_{\lambda_m}(v_m)[PZ_m^j],
\quad j=1,2.
\end{equation}

Recall $W_{m}$, $PZ_{m}^{j}$ and $Z_{m}^{j}$ defined in \eqref{eq:bubble}, \eqref{eq:PZm12} and \eqref{eq: definition of Zm012}, respectively, we derive 
$$-\Delta PZ_{m}^{j}=|x|^{2}e^{W_m}Z_{m}^{j}\quad\hbox{in} \quad B_1, \quad PZ_{m}^{j}=0\quad \hbox{  on}\quad \partial B_1 \quad \hbox{ for}\quad j=1,2.$$ It follows from the definition of $i_{p}^{\ast}$ that
\[
PZ_m^j=i_p^*\bigl(|x|^2e^{W_m}Z_m^j\bigr),
\quad j=1,2.
\]
Hence
\begin{equation}\label{eq:Ds lambda vm}
DS_{\lambda_m}(v_m)[PZ_m^j]
=
i_p^*\Bigl[
|x|^2e^{W_m}Z_m^j
-
\lambda_m V(x)|x|^2e^{v_m}PZ_m^j
\Bigr].
\end{equation}
Substituting \eqref{eq:Ds lambda vm} into \eqref{eq:definition of Rjm}, we have
\begin{equation}\label{eq:Rjm difference}
R_{j,m}
=
\Pi_m^\perp i_p^*\Bigl[
|x|^2e^{W_m}Z_m^j
-
\lambda_m V(x)|x|^2e^{v_m}PZ_m^j
\Bigr],
\quad j=1,2.
\end{equation}
We decompose the difference in \eqref{eq:Rjm difference} as
\begin{align*}
&\quad |x|^2 e^{W_m} Z_m^j
-
\lambda_m V(x)|x|^2 e^{v_m} PZ_m^j  \\
&=
|x|^2 e^{W_m}\bigl(Z_m^j-PZ_m^j\bigr)
+
\left(
|x|^2 e^{W_m}
-
\lambda_m V(x)|x|^2 e^{v_m}
\right)PZ_m^j .
\end{align*}
Accordingly, we write
\begin{equation}\label{eq:Rjm}
R_{j,m}
=
R_{j,m}^{(1)}
+
R_{j,m}^{(2)},
\end{equation}
where
\[
R_{j,m}^{(1)}
=
\Pi_m^\perp i_p^*
\left[
|x|^2 e^{W_m}\bigl(Z_m^j-PZ_m^j\bigr)
\right],
\]
and
\[
R_{j,m}^{(2)}
=
\Pi_m^\perp i_p^*
\left[
\left(
|x|^2 e^{W_m}
-
\lambda_m V(x)|x|^2 e^{v_m}
\right)PZ_m^j
\right].
\]

We first estimate $R_{j,m}^{(1)}$. Set
\[
f_{j,m}^{(1)}
:=
|x|^2 e^{W_m}\bigl(Z_m^j-PZ_m^j\bigr),
\quad j=1,2.
\]
By \eqref{eq:boundedness of ipstar} and the projection property, we deduce
\begin{equation}\label{eq:Rjm1}
    \|R_{j,m}^{(1)}\|\leq \|i_{p}^{\ast}(f_{j,m}^{(1)})\|\leq c_{p}\|f_{j,m}^{(1)}\|_{p}.
\end{equation}
Using \eqref{eq:PZm12}, we have
\[
\|f_{j,m}^{(1)}\|_{p}
\leq
C\lambda_{m}^{1/2}
\bigl\||x|^2e^{W_m}\bigr\|_{p},
\]
which, together with \eqref{eq:Lp estimate in Zhang}, we derive
\begin{equation*}
    \|f_{j,m}^{(1)}\|_{p}\leq C\lambda_{m}^{\frac{1}{2}-\frac{p-1}{2p}}\le\lambda _{m}^{1/4+\eta_1}o\left(\frac{1}{|\log\lambda_m|}\right),
\end{equation*}
where $\eta_1>0$ is small and $p>1$ is sufficiently close to $1$ such that $\frac{1}{2p}>\frac{1}{4}+\eta_1$.

Inserting the above estimate into \eqref{eq:Rjm1}, we arrive at
\begin{equation}\label{eq:Rjm1 estimate}
\|R_{j,m}^{(1)}\|
=\lambda _{m}^{\frac{1}{4}+\eta_1}
o\left(\frac{1}{|\log\lambda_m|}\right).
\end{equation}

We proceed to estimate $R_{j,m}^{(2)}$. Let
\[
f_{j,m}^{(2)}
:=
\left(
|x|^{2} e^{W_m}
-
\lambda_m V(x)|x|^{2} e^{v_m}
\right)
PZ_m^j .
\]
In view of \eqref{eq:PZm12} and \eqref{eq:Z012}, we get $\|PZ_{m}^{j}\|_{L^{\infty}}\leq C$, for $j=1,2$. 

Similarly to \eqref{eq:Rjm1}, we obtain
\begin{equation}\label{eq:Rjm2 estimate1}
    \|R_{j,m}^{(2)}\|
\leq
c_{p} \|f_{j,m}^{(2)}\|_{p}\leq C\left\|
|x|^2 e^{W_m}
-
\lambda_m V(x)|x|^2 e^{v_m}
\right\|_{p}.
\end{equation}

Recall that the expansion \eqref{eq:Pw expand},
\begin{equation*}
PW_m=W_m - \log\lambda_m+8\pi \mathcal H(x^2,b_m)+O(\lambda_m) .
\end{equation*}
We denote the remainder by $r_{m}(x)$, then
\begin{equation}\label{eq:rm estimate}
\|r_m\|_{L^\infty(B_1)} \leq C\lambda_m.
\end{equation}
By
$v_m = PW_m+\phi_m$, we derive
\[
\lambda_m e^{v_m}
=
e^{
W_m
+
8\pi \mathcal{H}(x^2,b_m)
+
r_m
+
\phi_m
}.
\]
Therefore,
\begin{equation}\label{eq:fjm2 difference}
 |x|^{2}e^{W_m}-\lambda_m V(x)|x|^{2}e^{v_m} =
|x|^{2}e^{W_m}\left(1-V(x)e^{8\pi \mathcal{H}(x^2,b_m)}e^{r_m+\phi_m}\right).
\end{equation}

Next we deal with the term $e^{8\pi\mathcal{H}(x^{2},b_{m})}$. By \eqref{H(x,y)}, one has
\[
\mathcal{H}(x^2,b_m)
=
\frac{1}{4\pi}
\log\left(
|x^2|^2 |b_m|^2 + 1 - 2\langle x^2,b_m\rangle
\right)=
\frac{1}{4\pi}\log A_m(x),
\]
where $
A_m(x)
:=
|x^2|^2 |b_m|^2 + 1 - 2\langle x^2,b_m\rangle$. Then we  obtain $$e^{8\pi \mathcal{H}(x^2,b_m)}=A_m(x)^2.$$ Taking into account that the bounds $|b_{m}|\leq M\sqrt{\lambda_m}$ and $|x|\leq1$, we derive
\[
\begin{aligned}
|A_m(x)-1|
&=
\left|
|x^2|^2 |b_m|^2
-
2\langle x^2,b_m\rangle
\right|
\\
&\le
|x^2|^2 |b_m|^2
+
2|x^2||b_m|
\\
&\le
C|b_m|^2
+
C|b_m|\\
&\le C\sqrt{\lambda_m}.
\end{aligned}
\]
Using the uniform boundedness of $A_m(x)$, we get
\begin{equation}\label{eq:Am2-1}
|A_m(x)^2-1|
=
|A_m(x)-1|\,|A_m(x)+1|
\le
C\sqrt{\lambda_m},
\end{equation}
namely,
\[
|
e^{8\pi \mathcal{H}(x^2,b_m)}
-
1
|
\le
C\sqrt{\lambda_m}.
\]

We now estimate the right hand side of \eqref{eq:fjm2 difference}. Let
\[
B_m(x)
:=
1
-
V(x)e^{8\pi\mathcal{H}(x^2,b_m)}e^{r_m+\phi_m}=1
-
V(x)A_{m}(x)^{2}e^{r_m+\phi_m}.
\]
We decompose it as
\[
B_m
=
(1 - V A_m^2 e^{r_m})
+
V A_m^2 e^{r_m}(1 - e^{\phi_m}),
\]
then
\begin{equation}\label{eq:fjm2 expand}
|x|^2 e^{W_m} B_m
=
G_{m,1}+G_{m,2},
\end{equation}
where $G_{m,1}=|x|^2 e^{W_m}\left(1 - V A_m^2 e^{r_m}\right)$, $G_{m,2}
=|x|^2 e^{W_m} V A_m^2 e^{r_m}\left(1 - e^{\phi_m}\right).$

 We first estimate $G_{m,1}$. By \eqref{eq:V}, \eqref{eq:rm estimate} and \eqref{eq:Am2-1}, we derive
 \begin{align*}
\left|1 - V A_m^2 e^{r_m}\right|&=\left|(1-V)+V(1-A_{m}^{2})+VA_{m}^{2}(1-e^{r_m})\right|\\
&\le
|1-V|
+
C|1-A_m^2|
+
C|1-e^{r_m}|
\\
&\le
C|x|^2
+
C\sqrt{\lambda_m}
+
C|r_m|
\\
&\le
C|x|^2
+
C\sqrt{\lambda_m}.
 \end{align*}
Hence
\begin{equation}\label{eq:Gm1}
|G_{m,1}|
\leq
C|x|^4 e^{W_m}
+
C\sqrt{\lambda_m}\,|x|^2 e^{W_m}.
\end{equation}
By \eqref{eq:Gm1} and Lemma \ref{lem:zhangLp}, we obtain
\begin{equation}\label{eq:Gm1 Lp estimate}
\begin{aligned}
\|G_{m,1}\|_p&\le C\||x|^4 e^{W_m}
+
\sqrt{\lambda_m}\,|x|^2 e^{W_m}\|_p\\
&\le C\||x|^4 e^{W_m}\|_p
+
C\sqrt{\lambda_m}\,\||x|^2 e^{W_m}\|_p\\
&\le C\lambda_{m}^{\frac{1}{2}-\frac{p-1}{2p}}+C\lambda_{m}^{\frac{1}{2}-\frac{p-1}{2p}}\\
&=C\lambda_{m}^{\frac{1}{2p}}.
\end{aligned}
\end{equation}

In the following, we estimate $G_{m,2}$. Since \(V\), \(A_m\), and \(e^{r_m}\) are uniformly bounded, it follows that
\[
\|G_{m,2}\|_{p}
\le
C\||x|^2 e^{W_m}\bigl(1-e^{\phi_m}\bigr)\|_{p}.
\]
Then, by applying H${\rm \ddot{o}}$lder's inequality with $\frac{1}{q}+\frac{1}{q'}=1$, we derive
\[
\left\||x|^2 e^{W_m}\bigl(1-e^{\phi_m}\bigr)\right\|_{p}
\le
\left\||x|^2 e^{W_m}\right\|_{pq}
\left\|1-e^{\phi_m}\right\|_{pq'}.
\]
According to Lemma \ref{lem:zhangLp}, we get
\begin{equation}\label{eq:Gm2 pq estimate}
\left\||x|^2 e^{W_m}\right\|_{pq}
\le
C\lambda_m^{-\frac{pq-1}{2pq}}.
\end{equation}

Next we claim: for any $r\ge1$, it holds
\begin{equation}\label{eq: 1-e^phi Lp estimate}
    \|1-e^{\phi_m}\|_{r}\leq C_r\|\phi_m\|.
\end{equation}
Indeed, a direct computation yields
\[
e^{\phi_m(x)}-1
=
\phi_m(x)\int_0^1 e^{s\phi_m(x)}\,ds,
\]
then
\[
|e^{\phi_m}-1|
\le
|\phi_m|
\int_0^1 e^{s\phi_m}\,ds.
\]
By \eqref{eq:phim H01 estimate}, $\|\phi_m\|$ is uniformly bounded. It follows from Lemma \ref{lem:M-T inequality} that for $s\in[0,1]$
\[
\|e^{s\phi_m}\|_{L^{2r}}\leq C_r .
\]
By H${\rm \ddot{o}}$lder's inequality, Minkowski's integral inequality and the Sobolev embedding $H_0^1(B_1)\hookrightarrow L^{2r}(B_1)$, we obtain
\begin{equation}\label{eq:Gm2 pq' estimate}
    \begin{aligned}
    \|e^{\phi_m}-1\|_{{r}}&\leq C \left\||\phi_m|\int_0^1 e^{s\phi_m}\,ds\right\|_{{r}}\\
    &\leq C\|\phi_m\|_{{2r}}\left\|\int_0^1 e^{s\phi_m}\,ds\right\|_{{2r}}\\
    &\leq C\|\phi_m\|_{{2r}}\int_0^1\| e^{s\phi_m}\|_{{2r}}\,ds\\
    &\leq C_r\|\phi_m\|\\
    &\leq C\lambda_m^{\frac12-\varepsilon}.
\end{aligned}
\end{equation}
A combination of \eqref{eq:Gm2 pq estimate} and \eqref{eq:Gm2 pq' estimate} leads to
\begin{equation}\label{eq:Gm2 Lp estimate}
\|G_{m,2}\|_p
\leq
C\lambda_m^{-\frac{pq-1}{2pq}}\lambda_m^{\frac{1}{2}-\varepsilon}
=C\lambda_m^{\frac{1}{2}-\varepsilon-\frac{pq-1}{2pq}}.
\end{equation}
Combining \eqref{eq:fjm2 difference}, \eqref{eq:fjm2 expand}, \eqref{eq:Gm1 Lp estimate} and \eqref{eq:Gm2 Lp estimate}, we have
\begin{equation}\label{eq:x ew-lam v x e v}
    \left\|
|x|^2 e^{W_m}
-
\lambda_m V(x)|x|^2 e^{v_m}
\right\|_p
\le
C\lambda_m^{\frac{1}{2p}}
+
C\lambda_m^{\frac12-\varepsilon-\frac{pq-1}{2pq}}\leq C \lambda_m^{\frac{1}{4}+\eta_{2}}o\left(\frac{1}{|\log \lambda_m|}\right).
\end{equation}
Here we select $p,q$ sufficiently close to $1$, and $\varepsilon>0$ is chosen sufficiently small, hence we can choose a small number $\eta_{2}>0$ such that both exponents are strictly larger than $\frac{1}{4}+\eta_{2}$. Moreover, the extra positive powers of $\lambda_m$ imply that
\[
C\lambda_m^{\frac{1}{2p}}
+
C\lambda_m^{\frac12-\varepsilon-\frac{pq-1}{2pq}}\leq C \lambda_m^{\frac{1}{4}+\eta_{2}}o\left(\frac{1}{|\log \lambda_m|}\right),
\]
which proves \eqref{eq:x ew-lam v x e v}. Substituting this estimate into \eqref{eq:Rjm2 estimate1} gives
\begin{equation}\label{eq:Rjm2 estimate}
\|R_{j,m}^{(2)}\| \leq C \lambda_m^{\frac{1}{4}+\eta_{2}}o\left(\frac{1}{|\log \lambda_m|}\right).
\end{equation}
It follows from \eqref{eq:Rjm}, \eqref{eq:Rjm1 estimate} and \eqref{eq:Rjm2 estimate} that
\begin{equation}\label{eq:Rjm estimate}
\|R_{j,m}\|
\leq
\lambda_m^{\frac14+\eta}
o\left(\frac{1}{|\log \lambda_m|}\right),
\end{equation}
where $\eta=\min\{\eta_1,\eta_2\}>0.$

We shall use the following invertibility from \cite{D-W-Z2026}. Define
\[
L_m(\psi)
:=
\Pi_m^\perp i_p^*
\left(
\lambda_m V(x)|x|^2 e^{PW_{m}}\psi
\right)
-\psi,
\quad \psi\in K_m^{\perp}.
\]
Then, by Proposition 3.7 in \cite{D-W-Z2026}, $L_m$ is invertible and
\[
\|(L_m)^{-1}\|
\leq C|\log\lambda_m|.
\]
Recall the definition of $\widetilde{L}_{m}$ in \eqref{eq:Lm tuta}, we have
\[
\widetilde L_m=-(L_m+H_m),
\]
where
\[
H_m(\psi)
=
\Pi_m^\perp i_p^*
\left[
\lambda_m V(x)|x|^2 e^{PW_{m}}
\left(e^{\phi_m}-1\right)\psi
\right].
\]
By an analogous argument to the estimates above, we obtain
\[
\|H_m\|
=
o\left(\frac{1}{|\log\lambda_m|}\right).
\]
Therefore,
\[
\|(L_m)^{-1}H_m\|
\leq
\|(L_m)^{-1}\|\,\|H_m\|
\leq
C|\log\lambda_m|
\,o\left(\frac{1}{|\log\lambda_m|}\right)
=o(1).
\]
In particular, for \(m\) large enough,
\[
\|(L_m)^{-1}H_m\|<1.
\]
By \cite[Theorem 4.40]{R-Y2008}, we obtain \(I+(L_m)^{-1}H_m\) is invertible.

Since $\widetilde L_m
=-
L_m
\left(
I+(L_m)^{-1}H_m
\right)$, we conclude that \(\widetilde L_m\) is invertible and 

$$
(\widetilde L_m)^{-1}
=-
\left(
I+(L_m)^{-1}H_m
\right)^{-1}
(L_m)^{-1},
$$ which means
\begin{equation}\label{eq:Lm tuta -1}
    \|(\widetilde L_m)^{-1}\|
\leq C|\log\lambda_m|.
\end{equation}

Finally, by \eqref{eq:Lm tuta psi m star}, \eqref{eq:Rjm}, \eqref{eq:Rjm estimate} and \eqref{eq:Lm tuta -1}, we conclude that there exists $\eta>0$ such that
\begin{align*}
    \|\psi_{m}^{\ast}\|&\leq C |\log\lambda_m|\sum\limits_{j=1}^{2}|c_{j,m}\|R_{j,m}\|\\
    &\leq C\lambda_m^{\frac14+\eta}(|c_{1,m}|+|c_{2,m}|)\\
    &=o(1)(|c_{1,m}|+|c_{2,m}|).
\end{align*}
\end{proof}
\begin{lemma}\label{lem:c1m=c2m=0}
    Assume that \eqref{eq:hessian V is nondegenerate} holds. Then for $m$ sufficiently large,
\[
c_{1,m}=c_{2,m}=0.
\]
\end{lemma}
\begin{proof}
    Recall the two Pohozaev identities in Lemma \ref{lem:Pohozaev identity}, we introduce two linear functionals
\begin{align*}
\mathcal{P}_{1,m}(\eta) := & \int_{B_1} \lambda_m e^{v_m} \eta \bigl( V_{x_1} x_1 - V_{x_2} x_2 \bigr) \, dx \nonumber
- \int_{\partial B_1} \partial_\nu v_m  \partial_\nu \eta \, (x_1^2 - x_2^2) \, d\sigma \nonumber \\
& + 2\pi v_{m,x_1}(0)  \eta_{x_1}(0) - 2\pi  v_{m,x_2}(0)  \eta_{x_2}(0)\\
=&\mathcal{P}_{1,m}^{\text{vol}}(\eta)+\mathcal{P}_{1,m}^{\text{bd}}(\eta)+\mathcal{P}_{1,m}^{\text{loc}}(\eta),
\end{align*}
where
\[\mathcal{P}_{1,m}^{\text{vol}}(\eta)=\int_{B_1} \lambda_m e^{v_m} \eta \bigl( V_{x_1} x_1 - V_{x_2} x_2 \bigr) \, dx,\]
\[\mathcal{P}_{1,m}^{\text{bd}}(\eta)=- \int_{\partial B_1} \partial_\nu v_m  \partial_\nu \eta \, (x_1^2 - x_2^2) \, d\sigma,\]
\[\mathcal{P}_{1,m}^{\text{loc}}(\eta)=2\pi v_{m,x_1}(0)  \eta_{x_1}(0) - 2\pi  v_{m,x_2}(0)  \eta_{x_2}(0).\]
Likewise
\begin{align*}
\mathcal{P}_{2,m}(\eta) := & \int_{B_1} \lambda_m e^{v_m} \eta \bigl( V_{x_1} x_2 + V_{x_2} x_1 \bigr) \, dx \nonumber
- 2\int_{\partial B_1} \partial_\nu v_m  \partial_\nu \eta \, x_1 x_2 \, d\sigma \nonumber \\
& + 2\pi v_{m,x_1}(0)  \eta_{x_2}(0) + 2\pi  v_{m,x_2}(0)  \eta_{x_1}(0)\\
=&\mathcal{P}_{2,m}^{\text{vol}}(\eta)+\mathcal{P}_{2,m}^{\text{bd}}(\eta)+\mathcal{P}_{2,m}^{\text{loc}}(\eta),
\end{align*}
where
\[\mathcal{P}_{2,m}^{\text{vol}}(\eta)=\int_{B_1} \lambda_m e^{v_m} \eta \bigl( V_{x_1} x_2 + V_{x_2} x_1 \bigr) \, dx,\]
\[\mathcal{P}_{2,m}^{\text{bd}}(\eta)=- 2\int_{\partial B_1} \partial_\nu v_m  \partial_\nu \eta \, x_1 x_2 \, d\sigma,\]
\[\mathcal{P}_{2,m}^{\text{loc}}(\eta)=2\pi v_{m,x_1}(0)  \eta_{x_2}(0) + 2\pi  v_{m,x_2}(0)  \eta_{x_1}(0).\]
Then we have $\mathcal{P}_{1,m}(\psi_m)=\mathcal{P}_{2,m}(\psi_m)=0$. By \eqref{decompose psi m}, we obtain
\begin{equation}\label{eq:systerm for Pim}
    \sum_{j=1}^2 \mathcal{P}_{i,m}(PZ_m^j)c_{j,m} = -\mathcal{P}_{i,m}(\psi_m^{\ast}), \quad i=1,2.
\end{equation}

Next, we estimate $\mathcal{P}_{i,m}(\psi_m^{\ast})$. We first deal with the two volume integral terms $\mathcal{P}_{1,m}^{\text{vol}}(\psi_{m}^{\ast})$ and $\mathcal{P}_{2,m}^{\text{vol}}(\psi_{m}^{\ast})$. By \eqref{eq:V}, one has
\[
V_{x_1} x_1 - V_{x_2} x_2=O(|x|^{2}),\quad V_{x_1} x_2 + V_{x_2} x_1=O(|x|^{2}).
\]
By applying H${\rm \ddot{o}}$lder's inequality, Lemmas \ref{lem:zhangLp}, \ref{lem:M-T inequality} and \ref{lem:estimate for psi m star}, and Sobolev embedding, we derive:
\begin{equation}\label{eq:Pim vol psi m star}
    \begin{aligned}
|\mathcal{P}_{i,m}^{\text{vol}}(\psi_{m}^{\ast})|&\leq \left|\int_{B_1} \lambda_m e^{v_m}\psi_m^* O(|x|^2)\,dx
\right|\\
&\leq C\left\|\lambda_m |x|^2 e^{v_m}\right\|_{p}
\left\|\psi_m^*\right\|_{{p'}}  \\
&\leq C \left\|\lambda_m |x|^2 e^{PW_m}\right\|_{{pq}}\|e^{\phi_m}\|_{pq'}
\left\|\psi_m^*\right\|\\
&\leq C\lambda_{m}^{-\frac{pq-1}{2pq}}\lambda_{m}^{\frac{1}{4}+\eta}(|c_{1,m}|+|c_{2,m}|)\\
&=o(1)(|c_{1,m}|+|c_{2,m}|),\quad i=1,2,
\end{aligned}
\end{equation}
where  $\frac{1}{p}+\frac{1}{p'}=1$,  $\frac{1}{q}+\frac{1}{q'}=1$ and $p,q>1$ are chosen sufficiently close to 1.

In what follows, we estimate the boundary terms
\[\mathcal{P}_{i,m}^{\text{bd}}(\psi_{m}^{\ast})=\int_{\partial B_1} \partial_\nu v_m  \partial_\nu \psi_{m}^{\ast} \, Q_{i}(x) \, d\sigma,\]
where $Q_{1}(x)=x_{2}^{2}-x_{1}^{2}$, $Q_{2}(x)=-2x_{1}x_2$. Let
\[
g_{i,m}:=\frac{\partial v_{m}}{\partial\nu}Q_{i}\quad\text{on}\,\,\partial B_{1}.
\]
We claim that $g_{i,m}$ is uniformly bounded. 

Indeed, we take $$A_\delta=\{x\in \overline{B_1}:1-2\delta<|x|\leq1\},$$ which stays away from the blow-up point $0$. By $|b_m|=O(\sqrt{\lambda_m})$, \eqref{eq:Pw expand}, and the remainder of \eqref{eq:Pw expand} is harmonic and its boundary value is $O(\lambda)$ in $C^{2}(\partial B_1)$, then we deduce that $\|PW_m\|_{C^1(A_\delta)} \leq C.$ Since $v_m=PW_m+\phi_m$ and
\begin{equation}\label{eq:-laplace vm}
     -\Delta v_m = \lambda_m V(x)|x|^2 e^{v_m}.
\end{equation}
Lemma \ref{lem:M-T inequality} and \eqref{eq:phim H01 estimate} imply the right hand side of \eqref{eq:-laplace vm} is uniformly bounded in \(L^r(A_\delta)\) for some \(r>2\). Hence the standard boundary elliptic estimates and Morrey embedding yield
\[
\|v_m\|_{C^{1,\alpha}(A_{\delta/2})}\leq C,
\]
for some $\alpha\in(0,1).$ Applying the Schauder estimate \cite[Theorem 6.6]{G-T2001}, we have
\[
\|v_m\|_{C^{2,\alpha}(A_{\delta/4})}\leq C,
\]
which gives
\[
\|g_{i,m}\|_{C^1(\partial B_1)} \leq C.
\]
We proceed to choose an extension \(G_{i,m}\in C^1(\overline{B_1})\) such that
\(G_{i,m}=g_{i,m}\) on \(\partial B_1\), and $\operatorname{supp} G_{i,m}\subset A:=\left\{x:\frac12<|x|\leq1\right\}.$ Moreover, the extension can be chosen so that
\begin{equation}\label{eq:Gim W1 bounded}
   \|G_{i,m}\|_{W^{1,\infty}(B_1)}\leq C.
\end{equation}
Then
\[
\mathcal{P}_{i,m}^{\text{bd}}(\psi_{m}^{\ast})
=
\int_{\partial B_1}G_{i,m}\partial_\nu\psi_m^*\,d\sigma .
\]
Gauss formula gives
\begin{equation}\label{eq:green formula for bd term}
\int_{\partial B_1} G_{i,m}\partial_\nu \psi_m^*\,d\sigma
=
\int_{B_1}\nabla G_{i,m}\cdot \nabla \psi_m^*\,dx
+
\int_{B_1}G_{i,m}\Delta \psi_m^*\,dx .
\end{equation}
By \eqref{eq:Gim W1 bounded} and Lemma \ref{lem:estimate for psi m star}, the first term on the right hand side of \eqref{eq:green formula for bd term} satisfies
\[
\left|
\int_{B_1}\nabla G_{i,m}\cdot \nabla \psi_m^*\,dx
\right|
\leq
C\|\psi_m^*\|
\leq
C\lambda_{m}^{\frac{1}{4}+\eta}(|c_{1,m}|+|c_{2,m}|)= o(1)(|c_{1,m}|+|c_{2,m}|).
\]

On the other hand, recalling the definitions of $PZ_{m}^{j}$ and $Z_{m}^{j}$, together with \eqref{decompose psi m}, for $j=1,2$, we get
\begin{equation}\label{eq:equation for psi m star}
\Delta\psi_m^*
=
-q_m\psi_m^*
-
\sum_{j=1}^{2}c_{j,m}
\left(
q_{m}PZ_m^j
-
|x|^2e^{W_m}Z_m^j
\right),
\end{equation}
where $q_{m}=\lambda_m V(x)|x|^2e^{v_m}.$ Since $\operatorname{supp}G_{i,m}\subset A:=\{x:\frac12<|x|\le1\}$, which stays away from the blow-up point $0$, we have
\begin{equation}\label{eq:boundedness of lambda V |x|e^v and |x|e^W}
\|q_m\|_{L^p(A)}\leq C\lambda_m,
\quad
\left\||x|^2 e^{W_m}\right\|_{L^p(A)}
\leq C\lambda_m.
\end{equation}
From \eqref{eq:green formula for bd term}, \eqref{eq:boundedness of lambda V |x|e^v and |x|e^W}, Lemma \ref{lem:estimate for psi m star}, the uniform boundedness of $PZ_{m}^{j}$ and $Z_{m}^{j}$ and Sobolev embedding, we derive
\begin{equation}\label{eq:Gim lambda m one}
\begin{aligned}
\left|
\int_{B_1}G_{i,m}q_m\psi_m^*\,dx
\right|
&\leq
C\|q_m\|_{L^p(A)}
\|\psi_m^*\|_{{p'}}  \\
&\leq
C\lambda_m\|\psi_m^*\|
\leq
C\lambda_{m}^{\frac{5}{4}+\eta}(|c_{1,m}|+|c_{2,m}|)\\
&= o(1)(|c_{1,m}|+|c_{2,m}|),
\end{aligned}
\end{equation}
and
\begin{equation}\label{eq:Gim lambda m two}
\begin{aligned}
&\quad \left|
\int_{B_1}G_{i,m}
\left(
q_mPZ_m^j
-
|x|^2e^{W_m}Z_m^j
\right)\,dx
\right|  \\
&\leq
C\|q_m\|_{L^1(A)}
+
C\left\||x|^2e^{W_m}\right\|_{L^1(A)}
\leq
C\lambda_m .
\end{aligned}
\end{equation}
A combination of \eqref{eq:Gim lambda m one} and \eqref{eq:Gim lambda m two} yields that the second term on the right hand side of \eqref{eq:green formula for bd term} satisfies
\[
\left|
\int_{B_1}G_{i,m}\Delta \psi_m^*\,dx
\right|=o(1)(|c_{1,m}|+|c_{2,m}|).
\]
Hence we derive
\[
|\mathcal{P}_{i,m}^{\text{bd}}(\psi_{m}^{\ast})|=o(1)(|c_{1,m}|+|c_{2,m}|),\quad i=1,2.
\]

It remains to estimate the local derivative terms $\mathcal{P}_{i,m}^{\text{loc}}(\psi_{m}^{\ast})$ for $i=1,2$. By Proposition 2.4 in \cite{D-W-Z2026}, we have
\begin{equation}\label{eq:derivative of vm is o(1)}
\frac{\partial v_m}{\partial x_i}(0)=o(1),
\quad i=1,2,
\end{equation}
as \(m\to+\infty\). Then we only need to prove that
\begin{equation}\label{eq:derivative of psi m star is bounded}
    \left|\frac{\partial\psi_{m}^{\ast}}{\partial x_{i}}(0)\right|\leq C(|c_{1,m}|+|c_{2,m}|),\quad \text{for}\;\; i=1,2.
\end{equation}
We begin by rewriting \eqref{eq:equation for psi m star} as
\begin{equation}\label{eq:psi m equation}
-\Delta \psi_m^*
= q_m \psi_m^*
+ \sum_{j=1}^{2} c_{j,m} H_{j,m},
\end{equation}
where $H_{j,m}(x) := q_m P Z_m^j - |x|^2 e^{W_m} Z_m^j.$ We define $\widetilde{\psi}_m^*(y):=\psi_m^*(\rho_m y),$ by applying the change of variable $x=\rho_m y$, we obtain in \(B_2(0)\),
\begin{equation}\label{eq:equation for psi m star tuta}
-\Delta_y \widetilde{\psi}_m^*
=
\widetilde{q}_m(y)\widetilde{\psi}_m^*
+
\sum_{j=1}^{2}c_{j,m}\widetilde H_{j,m}(y),
\end{equation}
where $\widetilde{q}_m(y):=\rho_m^2 q_m(\rho_m y)$, $\widetilde H_{j,m}(y):=\rho_m^2 H_{j,m}(\rho_m y)$.

We next estimate $\widetilde{q}_{m}(y)$. By the expansion \eqref{eq:Pw expand}, we derive
\[
\rho_m^2 q_m(\rho_m y)
=
Q_m(y)e^{\phi_m(\rho_m y)}O(1),
\]
where $Q_m(y):=32V(\rho_m y)\frac{|y|^2}{\left(1+|y^2-\xi_m|^2\right)^2}$. It follows from Lemma \ref{lem:M-T inequality} that for each fixed $s>1$,
\begin{equation}\label{eq:e phim rmy Ls}
    \|e^{\phi_m(\rho_m \cdot)}\|_{L^s(B_2)}
=
\rho_m^{-2/s}
\|e^{\phi_m}\|_{L^s(B_{2\rho_m})}
\leq
C_s\rho_m^{-2/s},
\end{equation}
which implies
\begin{equation}\label{eq:q m tuta Ls estimate}
   \|\widetilde{q}_m\|_{L^s(B_2)} \leq C_s \rho_m^{-2/s}.
\end{equation}

In the following, we estimate $\widetilde H_{j,m}$ for $j=1,2$. We first decompose $\widetilde H_{j,m}$ as
\[
\begin{aligned}
\widetilde{H}_{j,m}(y)
&= \rho_m^2 |\rho_m y|^2 e^{W_m(\rho_m y)}
\left( PZ_m^j(\rho_m y) - Z_m^j(\rho_m y) \right) \\
&\quad + \rho_m^2
\left( q_m(\rho_m y) - |\rho_m y|^2 e^{W_m(\rho_m y)} \right)
PZ_m^j(\rho_m y)\\
&=\widetilde{H}_{j,m}^{(1)}(y)+\widetilde{H}_{j,m}^{(2)}(y).
\end{aligned}
\]
In view of the definition of $Q_{m}(y)$, we get $|\rho_m y|^2 e^{W_m(\rho_m y)}$=$\rho_{m}^{-2}\frac{Q_m(y)}{V(\rho_{m}y)}$. Then by \eqref{eq:PZm12}, for some $p_0>1$, one has
\begin{equation}\label{eq:Hjm 1 tuta Lp estimate}
\|\widetilde{H}_{j,m}^{(1)}(y)\|_{L^{p_0}(B_2)}=\left\|\frac{Q_m(y)}{V(\rho_{m}y)}\left( PZ_m^j(\rho_m y) - Z_m^j(\rho_m y) \right)\right\|_{L^{p_0}(B_2)}\leq C\rho_{m}^{2},
\end{equation}
where we used $\left\|\frac{Q_m(y)}{V(\rho_{m}y)}\right\|_{L^{\infty}(B_2)}\le C$. 

For $\widetilde{H}_{j,m}^{(2)}(y)$, a direct computation yields that
\begin{equation}\label{eq:H jm2 tuta}
    \widetilde{H}_{j,m}^{(2)}(y)=Q_m(y) PZ_m^j(\rho_m y)
\left(
E_m(y)e^{\phi_m(\rho_m y)}
- \frac{1}{V(\rho_m y)}
\right),
\end{equation}
where
\begin{equation}\label{eq:definition of Emy}
E_{m}(y):=e^{8\pi \mathcal{H}((\rho_m y)^2,b_m)+O(\lambda_m)}.
\end{equation}
We rewrite the inner factor in \eqref{eq:H jm2 tuta} as
\begin{equation}\label{eq:decompose Emy e phim rhom y}
    E_m(y)e^{\phi_m(\rho_m y)}- \frac{1}{V(\rho_m y)}=E_m(y)(e^{\phi_m(\rho_m y)}-1)
+(E_m(y)-1)+\left(1-\frac{1}{V(\rho_m y)}\right).
\end{equation}
On \(B_2\), we show that $
E_m(y)-1=O(\rho_m^2)$. Moreover, it follows from \eqref{eq:standard assumptions on V}-\eqref{eq:V0 nabla V0 and nondegenerate} that $$
1-\frac{1}{V(\rho_m y)}
=
O(\rho_m^2|y|^2)
=
O(\rho_m^2)
\quad \hbox{ uniformly for}\quad  y\in B_2.$$ Hence
\begin{equation}\label{eq:Em e^phim-1/V}
\left\|
E_m e^{\phi_m(\rho_m\cdot)}
-\frac{1}{V(\rho_m\cdot)}
\right\|_{L^{p_0}(B_2)}
\leq
C\|e^{\phi_m(\rho_m\cdot)}-1\|_{L^{p_0}(B_2)}
+C\rho_m^2.
\end{equation}
Using \eqref{eq: 1-e^phi Lp estimate} and the change of variable $x=\rho_m y$, we derive
\begin{equation}\label{eq:scaling for e^phi -1}
    \|e^{\phi_m(\rho_m\cdot)}-1\|_{L^{p_0}(B_2)}=\rho_{m}^{-2/p_0}\|e^{\phi_m}-1\|_{L^{p_0}(B_{2\rho_m)}}
\leq
C \rho_{m}^{-2/p_0}\|\phi_m\|=C \rho_{m}^{-2/p_0+2-4\varepsilon}.
\end{equation}
Substituting \eqref{eq:scaling for e^phi -1} into \eqref{eq:Em e^phim-1/V} gives
\begin{equation*}
\left\|
E_m e^{\phi_m(\rho_m\cdot)}
-\frac{1}{V(\rho_m\cdot)}
\right\|_{L^{p_0}(B_2)}
\leq
C \rho_{m}^{-2/p_0+2-4\varepsilon},
\end{equation*}
together with the $L^{\infty}$-boundedness of $Q_m(y)$ and $PZ_{m}^{j}$, we obtain
\begin{equation}\label{eq:H jm2 tuta Lp estimate}
    \|\widetilde{H}_{j,m}^{(2)}(y)\|_{L^{p_0}(B_2)}\le C\rho_{m}^{-2/p_0+2-4\varepsilon}.
\end{equation}
Combining \eqref{eq:Hjm 1 tuta Lp estimate} with \eqref{eq:H jm2 tuta Lp estimate}, we deduce
\begin{equation}\label{eq:H jm tuta Lp estimate}
    \|\widetilde{H}_{j,m}(y)\|_{L^{p_0}(B_2)}\le C\rho_{m}^{-2/p_0+2-4\varepsilon}.
\end{equation}

For \eqref{eq:equation for psi m star tuta}, the interior elliptic estimate \cite[Theorem 9.11]{G-T2001} shows
\begin{equation}\label{eq:interior estimate for psi m star tuta}
\|\widetilde{\psi}_m^*\|_{W^{2,p_0}(B_1)}
\le C(
\|\widetilde{\psi}_m^*\|_{L^{p_0}(B_2)}
+
\|\widetilde{q}_m\widetilde{\psi}_m^*\|_{L^{p_0}(B_2)}
+
\sum_{j=1}^{2}|c_{j,m}|\|\widetilde H_{j,m}\|_{L^{p_0}(B_2)}
).
\end{equation}
By \eqref{eq:psi m star H01 estimate} and Sobolev embedding, we obtain for $p_0>2$,
\begin{equation}\label{eq:psi m star tuta Lp estimate}
\|\widetilde{\psi}_m^{\ast}\|_{L^{p_0}(B_2)}=
\rho_m^{-2/p_0}
\|\psi_m^*\|_{L^{p_0}(B_{2\rho_m})}\le
C\rho_m^{-2/p_0}
\|\psi_m^*\|
=C \rho_{m}^{1+4\eta-\frac{2}{p_0}}(|c_{1,m}|+|c_{2,m}|).
\end{equation}
By applying H${\rm \ddot{o}}$lder's inequality, we derive
\begin{equation}\label{eq:qm tuta psi m star tuta Lp estimate}
\|\widetilde{q}_m\widetilde{\psi}_m^*\|_{L^{p_0}(B_2)}\leq\|\widetilde{q}_m\|_{L^s(B_2)}\|\widetilde{\psi}_m^*\|_{L^r(B_2)},
\end{equation}
where $\frac{1}{s}+\frac{1}{r}=\frac{1}{p_0}$. Inserting \eqref{eq:q m tuta Ls estimate} and \eqref{eq:psi m star tuta Lp estimate} into \eqref{eq:qm tuta psi m star tuta Lp estimate}, we get
\begin{equation}\label{eq:qm tuta psi m star tuta Lp estimate1}
\|\widetilde{q}_m\widetilde{\psi}_m^*\|_{L^{p_0}(B_2)}\leq C \rho_{m}^{1+4\eta-\frac{2}{r}-\frac{2}{s}}(|c_{1,m}|+|c_{2,m}|).
\end{equation}
Substituting \eqref{eq:H jm tuta Lp estimate}, \eqref{eq:psi m star tuta Lp estimate} and \eqref{eq:qm tuta psi m star tuta Lp estimate1} into \eqref{eq:interior estimate for psi m star tuta} gives
\begin{equation*}
\begin{aligned}
\|\widetilde{\psi}_m^*\|_{W^{2,p_0}(B_1)}
&\le
C\left(
\rho_m^{1+4\eta-\frac{2}{p_0}}
+
\rho_m^{1+4\eta-\frac{2}{r}-\frac{2}{s}}
+
\rho_m^{2-4\varepsilon-\frac{2}{p_0}}
\right)
\left(|c_{1,m}|+|c_{2,m}|\right)\\
&\leq C \rho_m^{1+4\eta-\frac{2}{p_0}}
\left(|c_{1,m}|+|c_{2,m}|\right),
\end{aligned}
\end{equation*}
where $4\varepsilon<1-4\eta$. By Morrey embedding $W^{2,p_0}(B_{1}) \hookrightarrow C^{1,\alpha}(B_{1})$, $\alpha=1-2/p_0>0$, we derive
\begin{equation}\label{eq:nabla y psi m star tuta 0}
|\nabla_y \widetilde{\psi}_m^*(0)|
\le\|\nabla\widetilde{\psi}_{m}^{\ast}\|\leq
C \rho_m^{1+4\eta-\frac{2}{p_0}}
\left(|c_{1,m}|+|c_{2,m}|\right).
\end{equation}
Scaling back, we have
\[
|\nabla_x \psi_m^*(0)|
\le|\rho_m^{-1}\nabla_y \widetilde{\psi}_m^*(0)|\le
C\rho_m^{4\eta-\frac{2}{p_0}}\left(|c_{1,m}|+|c_{2,m}|\right).
\]
Since $\rho_m=\left(\frac{\lambda_m}{32}\right)^{1/4}$, we choose $p_0$ sufficiently large such that $4\eta-\frac{2}{p_0}>0$.  Therefore,
\[
|\nabla_x\psi_m^*(0)|
\le
C\left(|c_{1,m}|+|c_{2,m}|\right),
\]
which is \eqref{eq:derivative of psi m star is bounded}. By \eqref{eq:derivative of vm is o(1)} and \eqref{eq:derivative of psi m star is bounded}, we conclude
\[
|\mathcal{P}_{i,m}^{\text{loc}}(\psi_{m}^{\ast})|=o(1)(|c_{1,m}|+|c_{2,m}|),\quad i=1,2.
\]

We now estimate the main terms $\mathcal{P}_{i,m}(PZ_m^j)$, $i,j=1,2.$ As in the previous discussion, we decompose  $\mathcal{P}_{i,m}(PZ_m^j)$ as
\begin{equation}\label{eq:Pim PZmj decompose}
\mathcal{P}_{i,m}(PZ_m^j)
=
\mathcal{P}_{i,m}^{\mathrm{vol}}(PZ_m^j)
+
\mathcal{P}_{i,m}^{\mathrm{bd}}(PZ_m^j)
+
\mathcal{P}_{i,m}^{\mathrm{loc}}(PZ_m^j).
\end{equation}

We first prove that the boundary terms $\mathcal{P}_{i,m}^{\mathrm{bd}}(PZ_m^j),\;(i,j=1,2)$ are negligible. Recall that \eqref{eq: definition of Zm012}, we have
\[
Z_m^1(x)=
\frac{\rho_m^2 \operatorname{Re}(x^2-b_m)}
{\rho_m^4+|x^2-b_m|^2},
\quad
Z_m^2(x)=
\frac{\rho_m^2 \operatorname{Im}(x^2-b_m)}
{\rho_m^4+|x^2-b_m|^2}.
\]
Since \(b_m\to0\) for \(m\) large enough, we have
\[
|x^2-b_m|\geq c>0
\quad \text{for all } x\in \partial B_1,
\]
then we obtain
\[
\|Z_m^j\|_{C^1(\partial B_1)}\leq C\rho_m^2,
\quad j=1,2.
\]
In particular,
\[
Z_m^j=O(\rho_m^2),
\qquad
\partial_\nu Z_m^j=O(\rho_m^2)
\quad \text{on } \partial B_1.
\]
Set
\begin{equation}\label{eq:definition of hjm}
    h_m^j:=PZ_m^j-Z_m^j.
\end{equation}
By the definition of the projection \(P\), we have
\begin{equation}\label{eq:Laplace hjm}
\Delta h_m^j=0 \quad \text{in } B_1,
\quad
h_m^j=-Z_m^j \quad \text{on } \partial B_1.
\end{equation}
Therefore, by the standard boundary estimates for harmonic functions,
\begin{equation}\label{eq:hjm C1 bounded}
\|h_m^j\|_{C^1(\overline{B_1})}
\leq C\|Z_m^j\|_{C^1(\partial B_1)}
\leq C\rho_m^2.
\end{equation}
Consequently,
\begin{equation}\label{eq:partial PZmj}
\partial_\nu(PZ_m^j)
=
\partial_\nu Z_m^j+\partial_\nu h_m^j
=
O(\rho_m^2)
\quad \text{on } \partial B_1.
\end{equation}

We proceed to show that $\partial_\nu v_m=O(1)$ on $\partial B_1$. Indeed, it follows from \eqref{eq:Pw expand} that
\begin{equation}\label{eq:partial PW bounded}
    \partial_\nu PW_m=O(1) \quad \text{on}\quad \partial B_1.
\end{equation}
It remains to show that $\partial_\nu \phi_m=O(1)$ on $\partial B_1$. Recall that \eqref{eq:correction term phi m equation}, we take $A_{r_{0}}:=\overline{B_1}\setminus B_{r_0}$, $0<r_0<1$. Then we can obtain the second term on the right hand side of \eqref{eq:correction term phi m equation} is bounded. By Lemma \ref{lem:M-T inequality}, we derive
\[
\left\|\lambda_m V(x)|x|^2 e^{P W_m}\left(e^{\phi_m}-1\right)\right\|_{L^q(A_{r_0})} \le C,
\]
where $q>2$. Arguing as in the proof of Lemma \ref{lem:the convergence of psi m tuta} above, we obtain for $r_{0}<r_{1}<1$
\[
\|\nabla\phi_m\|_{L^{\infty}(A_{r_{1}})}\leq C,
\]
which implies
\begin{equation}\label{eq:eq:partial phi m bounded}
    \partial_\nu \phi_m = O(1) \quad \text{on} \quad\partial B_1.
\end{equation}
By \eqref{eq:partial PZmj}, \eqref{eq:partial PW bounded} and \eqref{eq:eq:partial phi m bounded}, we get
\begin{equation}\label{eq:Pbd is o(1)}
    \mathcal{P}_{i,m}^{\mathrm{bd}}(PZ_m^j)=O(\rho_{m}^{2})=o(1),\quad\text{for} \quad i,j=1,2.
\end{equation}

Now we estimate the local term in \eqref{eq:Pim PZmj decompose}. By \eqref{eq:Laplace hjm} and \eqref{eq:hjm C1 bounded}, we derive
\begin{equation}
    |\nabla h_{j}^{m}(0)|\leq C \rho_{m}^{2},
\end{equation}
which, together with \eqref{eq:definition of hjm}, we arrive at
\begin{equation}\label{eq:nabla PZmj 0 is o(1)}
    |\nabla PZ_{m}^{j}(0)|=o(1),\quad j=1,2.
\end{equation}
A combination of \eqref{eq:derivative of vm is o(1)} and \eqref{eq:nabla PZmj 0 is o(1)} leads to
\begin{equation}\label{eq:Ploc is o(1)}
    \mathcal{P}_{i,m}^{\mathrm{loc}}(PZ_m^j)=O(\rho_{m}^{2})=o(1), \quad i,j=1,2.
\end{equation}
Substituting \eqref{eq:Pbd is o(1)} and \eqref{eq:Ploc is o(1)} into \eqref{eq:Pim PZmj decompose}, we obtain
\begin{equation}\label{eq:Pim PZmj=Pvol}
\mathcal{P}_{i,m}(PZ_m^j)=\mathcal{P}_{i,m}^{\mathrm{vol}}(PZ_m^j)
+o(1), \quad i,j=1,2.
\end{equation}

We first consider $\mathcal{P}_{1,m}^{\mathrm{vol}}(PZ_m^j)$ for $j=1,2$. By \eqref{eq:V}, we have
\begin{equation*}
\mathcal{P}_{1,m}^{\mathrm{vol}}(PZ_m^j)
=
\int_{B_1}
\lambda_m e^{v_m} PZ_m^j
\left(\gamma_1 x_1^2-\gamma_2 x_2^2\right)\,dx
+o(1).
\end{equation*}
By applying the change of variable $x=\rho_{m}y$ and recalling the definition of  $\xi_m$, we deduce
\begin{equation}\label{eq:Pvol PZmj}
\mathcal{P}_{1,m}^{\mathrm{vol}}(PZ_m^j)
= 32\int_{B_{1/\rho_m}}
\frac{(\gamma_1 y_1^2-\gamma_2 y_2^2)\zeta_{mj}}
{(1+|\zeta_m|^2)^3}E_m(y)e^{\phi_m(\rho_my)}\,dy
+o(1),
\end{equation}
where $E_m(y)$ is defined in \eqref{eq:definition of Emy} and $\zeta_m:=y^{2}-\xi_m$. We are going to show that
\begin{equation}\label{eq:Emy e phim-1}
    \int_{B_{1/\rho_m}}
\frac{(\gamma_1 y_1^2-\gamma_2 y_2^2)\zeta_{mj}}
{(1+|\zeta_m|^2)^3}(E_m(y)e^{\phi_m(\rho_my)}-1)\,dy=o(1).
\end{equation}
In fact, we decompose \eqref{eq:Emy e phim-1} as
\begin{equation*}
\begin{aligned}
 &\quad \int_{B_{1/\rho_m}}
\frac{(\gamma_1 y_1^2-\gamma_2 y_2^2)\zeta_{mj}}
{(1+|\zeta_m|^2)^3}(E_m(y)e^{\phi_m(\rho_my)}-1)\,dy\\
&=\int_{B_{1/\rho_m}}
K_{m,j}E_m(y)(e^{\phi_m(\rho_my)}-1)\,dy+\int_{B_{1/\rho_m}}
K_{m,j}(E_m(y)-1)\,dy,
\end{aligned}
\end{equation*}
where $K_{m,j}=\frac{(\gamma_1 y_1^2-\gamma_2 y_2^2)\zeta_{mj}}
{(1+|\zeta_m|^2)^3}$. A direct calculation shows that
\begin{equation}\label{eq:decay of Kmj}
    |K_{m,j}|\leq C\frac{|y|^{2}(1+|y|^{2})}{(1+|y|^{4})^{3}}\leq\frac{C}{1+|y|^{8}}\in L^{p}(\mathbb{R}^{2}),\quad p>1.
\end{equation}
Thanks to \eqref{eq:Gm2 pq' estimate}, \eqref{eq:phi m tuta to 0}, \eqref{eq:decay of Kmj} and the $L^{\infty}$-boundedness of $E_m(y)$, we obtain
\[
\begin{aligned}
&\quad\quad\left|
\int_{B_{1/\rho_m}}K_{m,j}E_m(y)
(e^{\phi_m(\rho_m y)}-1)\,dy
\right| \\
&\quad\le
C\|K_{m,j}\|_{L^p(\mathbb R^2)}
\|e^{\phi_m(\rho_m\cdot)}-1\|_{L^{p'}(B_{1/\rho_m})} \\
&\quad\le
C\|K_{m,j}\|_{L^p(\mathbb R^2)}
\rho_m^{-2/p'}
\|e^{\phi_m}-1\|_{p'} \\
&\quad\le
C\rho_m^{-2/p'}\|\phi_m\| \\
&\quad\le
C\rho_m^{2-4\varepsilon-\frac2{p'}}
=o(1),
\end{aligned}
\]
provided that \(p'\) is sufficiently large and \(\varepsilon>0\) is sufficiently small. For the second term, a direct computation gives
\begin{equation*}
    |E_m(y)-1|\leq C\rho_m^{2},\quad y\in B_{1/\rho_m}.
\end{equation*}
By \eqref{eq:decay of Kmj}, we obtain
\[
\left|\int_{B_{1/\rho_m}} K_{m,j}\bigl(E_m(y)-1\bigr)\,dy\right|
\leq \|E_m-1\|_{L^\infty(B_{1/\rho_m})}
\int_{B_{1/\rho_m}} |K_{m,j}|\,dy
\leq C\rho_m^2=o(1).
\]
Therefore \eqref{eq:Pvol PZmj} becomes
\begin{equation*}
\mathcal{P}_{1,m}^{\mathrm{vol}}(PZ_m^j)
= 32\int_{B_{1/\rho_m}}
\frac{(\gamma_1 y_1^2-\gamma_2 y_2^2)\zeta_{mj}}
{(1+|\zeta_m|^2)^3}\,dy
+o(1),\quad j=1,2.
\end{equation*}
By Theorem A, 
\[
\xi_m \to \xi=(\xi^{*},0)\in\mathbb{R}^{2}.
\]
For convenience, set 
\[
s:=\xi^{*}\in\mathbb{R}.
\]
Making the change of variables $\zeta=y^{2}-\xi$, we obtain
\begin{align*}
\mathcal{P}_{1,m}^{\mathrm{vol}}(PZ_m^j)
&=
32\int_{\mathbb R^2}
\frac{1}{2}
\left[
(\gamma_1-\gamma_2) r_{s}(\zeta)
+
(\gamma_1+\gamma_2)(\zeta_1+s)
\right]
\frac{\zeta_j}{(1+|\zeta|^2)^3}
\frac{1}{2r_{s}(\zeta)}\,d\zeta
+o(1) \notag \\
&=
8\int_{\mathbb R^2}
\left[
(\gamma_1-\gamma_2)
+
(\gamma_1+\gamma_2)\frac{\zeta_1+s}{r_{s}(\zeta)}
\right]
\frac{\zeta_j}{(1+|\zeta|^2)^3}\,d\zeta
+o(1),
\end{align*}
where $r_{s}(\zeta)=\sqrt{(\zeta_1+s)^2+\zeta_2^2}.$ 

If $j=2$, by oddness we deduce $\mathcal{P}_{1,m}^{\mathrm{vol}}(PZ_m^2)=o(1)$, which means
\begin{equation}\label{eq:P1mPzm2}
    \mathcal{P}_{1,m}(PZ_m^2)=o(1).
\end{equation}
If $j=1$, we derive
\begin{equation*}
  \mathcal{P}_{1,m}^{\mathrm{vol}}(PZ_m^1)  =
8(\gamma_1+\gamma_2)
\int_{\mathbb R^2}
\frac{(\zeta_1+s)\zeta_1}
{r_{s}(\zeta)(1+|\zeta|^2)^3}\,d\zeta+o(1).
\end{equation*}
Then we have
\begin{equation}\label{eq:P1mPzm1}
    \mathcal{P}_{1,m}(PZ_m^1)=8(\gamma_1+\gamma_2)I_1(s)+o(1),
\end{equation}
where $I_1(s):=\int_{\mathbb R^2}\frac{(\zeta_1+s)\zeta_1}{r_{s}(\zeta)(1+|\zeta|^2)^3}\,d\zeta$.

Similarly, we also derive
\begin{equation*}
\mathcal{P}_{2,m}^{\mathrm{vol}}(PZ_m^j)=8(\gamma_1+\gamma_2)
\int_{\mathbb R^2}
\frac{\zeta_2\zeta_j}
{r_{s}(\zeta)(1+|\zeta|^2)^3}\,d\zeta
+o(1).
\end{equation*}
By oddness again we have
\begin{equation}\label{eq:P2mPzm1}
    \mathcal{P}_{2,m}(PZ_m^1)=o(1).
\end{equation}
Moreover, we get
\begin{equation}\label{eq:P2mPZm2}
    \mathcal{P}_{2,m}(PZ_m^2)=8(\gamma_1+\gamma_2)I_2(s)+o(1),
\end{equation}
where $I_2(s):=\int_{\mathbb R^2}\frac{\zeta_2^{2}}{r_{s}(\zeta)(1+|\zeta|^2)^3}\,d\zeta$.

Using \eqref{eq:P1mPzm2}, \eqref{eq:P1mPzm1}, \eqref{eq:P2mPzm1}, \eqref{eq:P2mPZm2}, one has
\begin{equation}\label{eq:matrix for PimPZjm}
\begin{pmatrix}
\mathcal P_{1,m}(PZ_m^1) & \mathcal P_{1,m}(PZ_m^2) \\
\mathcal P_{2,m}(PZ_m^1) & \mathcal P_{2,m}(PZ_m^2)
\end{pmatrix}
=
8(\gamma_1+\gamma_2)
\begin{pmatrix}
I_1(s) & 0 \\
0 & I_2(s)
\end{pmatrix}
+o(1).
\end{equation}
Inserting \eqref{eq:matrix for PimPZjm} into \eqref{eq:systerm for Pim}, we obtain
\[
\left[
8(\gamma_1+\gamma_2)
\begin{pmatrix}
I_1(s) & 0 \\
0 & I_2(s)
\end{pmatrix}
+o(1)
\right]
\begin{pmatrix}
c_{1,m} \\
c_{2,m}
\end{pmatrix}
=
o(1)(|c_{1,m}|+|c_{2,m}|).
\]
We can verify $I_{2}(s)>0$ and by Lemma 3.10 in \cite{D-W-Z2026}, $I_{1}(s)\neq0$. Moreover, it follows from \eqref{eq:hessian V is nondegenerate} that $\gamma_1+\gamma_2\ne0$. Hence $c_{1,m}=c_{2,m}=0$.
\end{proof}

We now give the proof of Theorem \ref{thm nondegenerate}.
\begin{proof}
 By Lemma \ref{lem:c1m=c2m=0} and \eqref{decompose psi m}, we deduce $\psi_m=\psi_{m}^{\ast}\in K_{m}^{\perp}$. Then Lemma \ref{lem:estimate for psi m star} yields
\[
\|\psi_m^*\|
\leq
C\lambda_m^{\frac14+\eta}
(|c_{1,m}|+|c_{2,m}|)=0.
\]
Consequently $\|\psi_m\|=0$, which contradicts the normalization $\|\psi_m\|=1$. Therefore, the kernel of $\widehat{L}_{m}$ is trivial, and the solution $v_m$ is non-degenerate.
\end{proof}

\section{Local uniqueness}
In this section, we prove Theorem \ref{thm:local uniqueness} via contradiction arguments. Suppose that $m\to+\infty$, $v_m^{(1)}\not\equiv v_m^{(2)}.$ Let
\[
\eta_m
:=
\frac{v_m^{(1)}-v_m^{(2)}}
{\|v_m^{(1)}-v_m^{(2)}\|}.
\]
Then $\|\eta_m\|=1$  and $\eta_m$ satisfies
\begin{equation}\label{eq:eta m equation}
    \begin{cases}
-\Delta \eta_m
=
\lambda_m V(x)|x|^2 C_m(x)\eta_m,
& x\in B_1,\\[2mm]
\eta_m=0,
& x\in \partial B_1,
\end{cases}
\end{equation}
where
\begin{equation}\label{eq:Cm}
C_m(x)
:=
\int_0^1
e^{t v_m^{(1)}(x)+(1-t)v_m^{(2)}(x)}
\,dt .
\end{equation}

Let
\[
\begin{cases}
-\Delta \bar{v}_{m}
=
\frac{\lambda_{m}}{2} V(x)|x|^2(e^{v_{m}^{(1)}}+e^{v_{m}^{(2)}}),
& x\in B_1,\\[2mm]
\bar{v}_{m}=0,
& x\in \partial B_1,
\end{cases}
\]
where $\bar{v}_{m}=\frac{v_{m}^{(1)}+v_{m}^{(2)}}{2}$. 

Repeating the argument used in Lemma \ref{lem:Pohozaev identity}, we have the following identities.
\begin{lemma}\label{lem:Pohozaev identity for eta}
It holds
\begin{equation*}
\begin{aligned}
&\lambda_{m}\int_{B_1} C_m(x)\eta_{m}
\left(
\frac{\partial V}{\partial x_1}x_1
-
\frac{\partial V}{\partial x_2}x_2
\right)\,dx  =\\
&\quad
\int_{\partial B_1}
\frac{\partial \bar{v}_m}{\partial \nu}
\frac{\partial \eta_m}{\partial \nu}
(x_1^2-x_2^2)\,d\sigma
-2\pi
\frac{\partial \bar{v}_{m}}{\partial x_1}(0)
\frac{\partial \eta_m}{\partial x_1}(0)
+
2\pi
\frac{\partial \bar{v}_{m}}{\partial x_2}(0)
\frac{\partial \eta_m}{\partial x_2}(0),
\end{aligned}
\end{equation*}
and
\begin{equation*}
\begin{aligned}
&\lambda_{m}\int_{B_1}C_m(x)\eta_m
\left(
\frac{\partial V}{\partial x_1}x_2
+
\frac{\partial V}{\partial x_2}x_1
\right)\,dx =\\
&\quad
2\int_{\partial B_1}
\frac{\partial \bar{v}_{m}}{\partial \nu}
\frac{\partial \eta_m}{\partial \nu}
x_1x_2\,d\sigma
-2\pi
\left[
\frac{\partial \bar{v}_{m}}{\partial x_1}(0)
\frac{\partial \eta_m}{\partial x_2}(0)
+
\frac{\partial \bar{v}_m}{\partial x_2}(0)
\frac{\partial \eta_m}{\partial x_1}(0)
\right].
\end{aligned}
\end{equation*}
\end{lemma}

Before rescaling, we first compare the two concentration
parameters. For $i=1,2$, set
\[
\xi_m^{(i)}:=\frac{b_m^{(i)}}{\rho_m^2},
\qquad
\rho_m=\left(\frac{\lambda_m}{32}\right)^{1/4}.
\]
By \eqref{eq:local uniqueness assumption}, the sequences $\{\xi_m^{(i)}\}$ are bounded. Hence,
after passing to a subsequence, we may assume that
\[
\xi_m^{(1)}\rightarrow \xi
\qquad\text{for some }\xi\in\mathbb R^2.
\]
Moreover, in view of \eqref{eq:bm1-bm2},
\[
|\xi_m^{(1)}-\xi_m^{(2)}|
=
\frac{|b_m^{(1)}-b_m^{(2)}|}{\rho_m^2}
=
O(\lambda_m^\sigma)
\longrightarrow 0.
\]
Therefore,
\[
\xi_m^{(2)}\longrightarrow \xi.
\]
By the classification of the limiting profile in \cite{D-W-Z2026},
we have
\[
\xi=(\xi^{*},0)
\qquad\text{for some }\xi^{*}\in\mathbb R.
\]
Thus, the two rescaled concentration parameters have the same limit.

 Let
$\widetilde{\eta}_{m}(y)=\eta_{m}(\rho_{m}y)$. Since
\[
\xi_m^{(i)}\longrightarrow\xi=(\xi^{*},0),
\qquad i=1,2,
\]
the same rescaling argument as in Lemma \ref{lem:the convergence of psi m tuta} yields the following
result.
\begin{lemma}\label{lem:the convergence of eta m tuta}
It holds
\begin{equation*}
\widetilde\eta_m
\longrightarrow
a_0Z_0+a_1Z_1+a_2Z_2,
\end{equation*}
uniformly in $C^1(B_R(0))$ for any $R>0$, where $a_{0}$, $a_1$ and $a_2$ are some constants and $Z_0$, $Z_1$, $Z_2$ are functions defined in \eqref{eq:Z012} with $c=\xi\in\mathbb{R}^{2}$.
\end{lemma}

To simplify the notation, we shall write
\[
b_m:=b_m^{(1)},
\qquad
\xi_m:=\xi_m^{(1)}
=\frac{b_m^{(1)}}{\rho_m^2},
\]
and
\[
W_m:=W_m^{(1)},
\qquad
Z_m^j:=Z_m^{j,(1)},
\qquad
PZ_m^j:=PZ_m^{j,(1)},
\qquad j=1,2.
\]
As in \eqref{decompose psi m}, we decompose
\begin{equation}\label{decompose eta m}
    \eta_{m}=d_{1,m}PZ_{m}^{1}+d_{2,m}PZ_{m}^{2}+\eta_{m}^{\ast},
\end{equation}
where $\eta_{m}^{\ast}\in K_{m}^{\perp}$, namely
\[
\int_{B_1}\nabla\eta_m^*\nabla PZ_m^j\,dx=0,
\quad j=1,2.
\]
It follows from Lemma \ref{lem:the convergence of eta m tuta} that $d_{1,m}$ and $d_{2,m}$ are bounded.
\begin{lemma}\label{lem:estimate for eta m star}
There exists $\sigma'>0$ such that
\begin{equation*}
\|\eta_{m}^{\ast}\|\leq C\lambda_{m}^{\frac{1}{4}+\sigma'}(|d_{1,m}|+|d_{2,m}|)= o(1)(|d_{1,m}|+|d_{2,m}|).
\end{equation*}
\end{lemma}
\begin{proof}
The proof follows a strategy similar to that used in Lemma \ref{lem:estimate for psi m star}. For the reader's convenience, we give the sketch of the proof.  We define
\begin{equation*}
    T_m^C(\eta):=\eta-i_p^*\bigl(\lambda_m V(x)|x|^2 C_m(x)\eta\bigr),
\end{equation*}
then it follows from \eqref{eq:eta m equation} that $T_m^C(\eta_m)=0.$ Substituting \eqref{decompose eta m} into this, we obtain
\begin{equation*}
\widetilde{L}_m^C \eta_m^*
=
-\sum_{j=1}^2 d_{j,m}R_{j,m}^{C},
\end{equation*}
where $\widetilde{\mathcal{L}}_m^C := \Pi_m^\perp T_m^C\big|_{K_m^\perp}$, \;$R_{j,m}^C := \Pi_m^\perp T_m^C(PZ_m^j),$ \;$j=1,2$. 

Proceeding as in \eqref{eq:Ds lambda vm}-\eqref{eq:Rjm}, we decompose $R_{j,m}^{C}$ as
\[
R_{j,m}^C
=R_{j,m}^{C,(1)}+R_{j,m}^{C,(2)},
\]
where
\begin{equation*}
R_{j,m}^{C,(1)}
=
\Pi_m^\perp i_p^*
\left[
|x|^2 e^{W_m}\bigl(Z_m^j-PZ_m^j\bigr)
\right],
\end{equation*}
and
\begin{equation*}
R_{j,m}^{C,(2)}
=
\Pi_m^\perp i_p^*
\left[
\bigl(|x|^2 e^{W_m}-\lambda_m V(x)|x|^2 C_m(x)\bigr)PZ_m^j
\right].
\end{equation*}
The estimate for $R_{j,m}^{C,(1)}$ is unchanged, and hence we obtain \begin{equation}\label{eq:Rjm1C estimate}
\|R_{j,m}^{C,(1)}\|
=\lambda _{m}^{\frac{1}{4}+\sigma_1}
o\left(\frac{1}{|\log\lambda_m|}\right).
\end{equation}

We next deal with $R_{j,m}^{C,(2)}$. By \eqref{eq:Cm} and \eqref{eq:Pw expand}, we deduce
\[|x|^2e^{W_m}-\lambda_m V(x)|x|^2C_m(x)
=
|x|^2e^{W_m}\bigl(1-V(x)\Theta_m(x)\bigr),\]
where we used
\begin{equation}\label{eq:lambda m Cm}
\lambda_m C_m(x)=e^{W_m}\Theta_m(x),
\end{equation}
and
\[
\Theta_m(x):=
\int_0^1
e^{(1-t)(W_m^{(2)}-W_m)+t(A_m^{(1)}+\phi_m^{(1)})+(1-t)(A_m^{(2)}+\phi_{m}^{(2)})}\,dt,\quad A_m^{(i)} := 8\pi H(x^2,b_m^{(i)})+ r_m^{(i)} .
\]
A direct computation shows
\[
\begin{aligned}
1-V(x)\Theta_m(x)
&=
(1-V(x))
+
V(x)\left(
1-
\int_0^1
e^{(1-t)(W_m^{(2)}-W_m)}
e^{tA_m^{(1)}+(1-t)A_m^{(2)}}\,dt
\right) \\
&\quad
+
V(x)\int_0^1
e^{(1-t)(W_m^{(2)}-W_m)}
e^{tA_m^{(1)}+(1-t)A_m^{(2)}}
\left(
1-e^{t\phi_m^{(1)}+(1-t)\phi_m^{(2)}}
\right)\,dt \\
&:=B_{m,0}+B_{m,1}+B_{m,2}.
\end{aligned}
\]
It follows from \eqref{eq:V} and Lemma \ref{lem:zhangLp} that
\begin{equation}\label{eq:Bm0 estimate}
    \bigl\||x|^2 e^{W_m} B_{m,0}\bigr\|_{p}
\leq C\lambda_m^{\frac{1}{2p}}.
\end{equation}
By \eqref{eq:bm1-bm2}, \eqref{eq:bubble} and $O(\lambda_m)$ estimate for $r_{m}^{(j)}$, we get
\begin{equation*}
    \|B_{m,1}\|_{\infty}=O(\lambda_{m}^{\frac{1}{4}+\epsilon}),
\end{equation*}
which indicates
\[
\bigl\||x|^2 e^{W_m} B_{m,1}\bigr\|_{p}
\leq C\lambda_m^{\frac{1}{4}+\epsilon-\frac{p-1}{2p}}.
\]
By \eqref{eq:local uniqueness assumption}, Lemmas \ref{lem:M-T inequality} and \ref{lem:zhangLp}, as well as H${\rm \ddot{o}}$lder's inequality, we deduce
\begin{equation}\label{eq:Bm2 estimate}
\bigl\||x|^2 e^{W_m} B_{m,2}\bigr\|_{p}
\leq
C\lambda_m^{\frac{1}{2}-\epsilon-\frac{pq-1}{2pq}}.
\end{equation}
Combining \eqref{eq:Bm0 estimate}-\eqref{eq:Bm2 estimate} and choosing $p,q\to1$, we obtain
\begin{equation*}
\bigl\||x|^2 e^{W_m}
-
\lambda_m V(x)|x|^2 C_m(x)
\bigr\|_{p}
=\lambda_{m}^{\frac{1}{4}+\sigma_2}
o\left(\frac{1}{|\log \lambda_m|}\right),
\end{equation*}
which yields
\begin{equation}\label{eq:Rjm2C estimate}
\|R_{j,m}^{C,(2)}\|
=\lambda _{m}^{\frac{1}{4}+\sigma_2}
o\left(\frac{1}{|\log\lambda_m|}\right).
\end{equation}
 A combination of \eqref{eq:Rjm1C estimate} and \eqref{eq:Rjm2C estimate} leads to
 \begin{equation*}
\|R_{j,m}^{C}\|
\leq
\lambda_m^{\frac14+\sigma'}
o\left(\frac{1}{|\log \lambda_m|}\right),
\end{equation*}
where $\sigma'=\min\{\sigma_1,\sigma_2\}>0.$

The invertibility analysis follows the same strategy as in Lemma \ref{lem:estimate for psi m star}. Hence we obtain
\begin{equation*}
    \|(\widetilde L_m^C)^{-1}\|
\leq C|\log\lambda_m|.
\end{equation*}
Therefore, we conclude that there exists $\sigma'>0$ such that
\begin{align*}
    \|\eta_{m}^{\ast}\|&\leq C |\log\lambda_m|\sum\limits_{j=1}^{2}|d_{j,m}\|R_{j,m}^{C}\|\\
    &\leq C\lambda_m^{\frac14+\sigma'}(|d_{1,m}|+|d_{2,m}|)\\
    &=o(1)(|d_{1,m}|+|d_{2,m}|).
\end{align*}
\end{proof}

\begin{lemma}\label{lem:d1m=d2m=0}
Assume that \eqref{eq:hessian V is nondegenerate} holds. Then for $m$ sufficiently large,
\[
d_{1,m}=d_{2,m}=0.
\]
\end{lemma}
\begin{proof}
     We begin by introducing two linear functionals related to the Pohozaev identities in Lemma \ref{lem:Pohozaev identity for eta}:
\begin{align*}
\mathcal{Q}_{1,m}(\eta) := & \int_{B_1} \lambda_m C_m(x) \eta \bigl( V_{x_1} x_1 - V_{x_2} x_2 \bigr) \, dx \nonumber
- \int_{\partial B_1} \partial_\nu \bar{v}_m  \partial_\nu \eta \, (x_1^2 - x_2^2) \, d\sigma \nonumber \\
& + 2\pi \bar{v}_{m,x_1}(0)  \eta_{x_1}(0) - 2\pi  \bar{v}_{m,x_2}(0)  \eta_{x_2}(0)\\
=&\mathcal{Q}_{1,m}^{\text{vol}}(\eta)+\mathcal{Q}_{1,m}^{\text{bd}}(\eta)+\mathcal{Q}_{1,m}^{\text{loc}}(\eta),
\end{align*}
where
\[\mathcal{Q}_{1,m}^{\text{vol}}(\eta)=\int_{B_1} \lambda_m C_m(x) \eta \bigl( V_{x_1} x_1 - V_{x_2} x_2 \bigr) \, dx,\]
\[\mathcal{Q}_{1,m}^{\text{bd}}(\eta)=- \int_{\partial B_1} \partial_\nu \bar{v}_m  \partial_\nu \eta \, (x_1^2 - x_2^2) \, d\sigma,\]
\[\mathcal{Q}_{1,m}^{\text{loc}}(\eta)=2\pi \bar{v}_{m,x_1}(0)  \eta_{x_1}(0) - 2\pi  \bar{v}_{m,x_2}(0)  \eta_{x_2}(0).\]
Similarly
\begin{align*}
\mathcal{Q}_{2,m}(\eta) := & \int_{B_1} \lambda_m C_m(x) \eta \bigl( V_{x_1} x_2 + V_{x_2} x_1 \bigr) \, dx \nonumber
- 2\int_{\partial B_1} \partial_\nu \bar{v}_m  \partial_\nu \eta \, x_1 x_2 \, d\sigma \nonumber \\
& + 2\pi \bar{v}_{m,x_1}(0)  \eta_{x_2}(0) + 2\pi  \bar{v}_{m,x_2}(0)  \eta_{x_1}(0)\\
=&\mathcal{Q}_{2,m}^{\text{vol}}(\eta)+\mathcal{Q}_{2,m}^{\text{bd}}(\eta)+\mathcal{Q}_{2,m}^{\text{loc}}(\eta),
\end{align*}
where
\[\mathcal{Q}_{2,m}^{\text{vol}}(\eta)=\int_{B_1} \lambda_m C_m(x) \eta \bigl( V_{x_1} x_2 + V_{x_2} x_1 \bigr) \, dx,\]
\[\mathcal{Q}_{2,m}^{\text{bd}}(\eta)=- 2\int_{\partial B_1} \partial_\nu \bar{v}_m  \partial_\nu \eta \, x_1 x_2 \, d\sigma,\]
\[\mathcal{Q}_{2,m}^{\text{loc}}(\eta)=2\pi \bar{v}_{m,x_1}(0)  \eta_{x_2}(0) + 2\pi  \bar{v}_{m,x_2}(0)  \eta_{x_1}(0).\]
Then we have $\mathcal{Q}_{1,m}(\eta_m)=\mathcal{Q}_{2,m}(\eta_m)=0$. Using \eqref{decompose eta m}, we get
\begin{equation*}
    \sum_{j=1}^2 \mathcal{Q}_{i,m}(PZ_m^j)d_{j,m} = -\mathcal{Q}_{i,m}(\eta_m^{\ast}), \quad i=1,2.
\end{equation*}
Set $\Psi_m:=PW_{m}^{(2)}-PW_m$. Then we have
\[
C_m(x)
=
\int_0^1 e^{t v_m^{(1)}(x)+(1-t)v_m^{(2)}(x)}\,dt
=
e^{PW_m(x)}F_m(x),
\]
where
\begin{equation}\label{eq:definition of Fmy}
F_m(x)=\int_0^1
e^{t\phi_m^{(1)}(x)+(1-t)(\Psi_m(x)+\phi_m^{(2)}(x))}\,dt.
\end{equation}
By \eqref{eq:local uniqueness assumption}, \eqref{eq:Pw expand} and \eqref{eq:bm1-bm2}, we obtain
\begin{equation}\label{eq:Psi m is o(1)}
    \|\Psi_m\|_{L^\infty(B_1)}=O(\lambda_{m}^{\sigma}).
\end{equation}
Moreover, by Lemma \ref{lem:M-T inequality} and $o(1)$ estimates of
$\phi_m^{(j)}$, for any fixed $r\geq 1$, we have
\[
\|F_m\|_{r}\leq C_r .
\]

We proceed to estimate $\mathcal{Q}_{i,m}(\eta_m^{\ast})$ for $i=1,2.$ We first consider the volume integral terms $\mathcal{Q}_{1,m}^{\text{vol}}(\eta_{m}^{\ast})$ and $\mathcal{Q}_{2,m}^{\text{vol}}(\eta_{m}^{\ast})$. Arguing as in the derivation of \eqref{eq:Pim vol psi m star}, we obtain
\begin{align*}
|\mathcal{Q}_{i,m}^{vol}(\eta_m^*)|
&\leq
C\int_{B_1}\lambda_m |x|^2 C_m(x)|\eta_m^*|\,dx  \\
&=
C\int_{B_1}\lambda_m |x|^2 e^{PW_m}F_m|\eta_m^*|\,dx  \\
&\leq
C\left\|\lambda_m |x|^2 e^{PW_m}F_m\right\|_{p}
\left\|\eta_m^*\right\|_{p'}  \\
&\leq
C\left\|\lambda_m |x|^2 e^{PW_m}\right\|_{pq}
\|F_m\|_{pq'}
\left\|\eta_m^*\right\|_{p'}  \\
&\leq
C\lambda_m^{-\frac{pq-1}{2pq}}\left\|\eta_m^*\right\|,\\
&\leq C\lambda^{\frac{1}{4}+\sigma'-\frac{pq-1}{2pq}}(|d_{1,m}|+|d_{2,m}|),\quad i=1,2,
\end{align*}
where  $\frac{1}{p}+\frac{1}{p'}=1$,  $\frac{1}{q}+\frac{1}{q'}=1$ and $p,q>1$ are chosen sufficiently close to 1.

Following the analysis of the boundary terms $\mathcal{P}_{i,m}^{\text{bd}}(\psi_{m}^{\ast})$ in Lemma \ref{lem:c1m=c2m=0}, we obtain
\[
|\mathcal{Q}_{i,m}^{\text{bd}}(\eta_{m}^{\ast})|=o(1)(|d_{1,m}|+|d_{2,m}|),\quad i=1,2.
\]

In the following, we estimate the local derivative terms $\mathcal{Q}_{i,m}^{\text{loc}}(\eta_{m}^{\ast})$ for $i=1,2$. In view of \eqref{eq:derivative of vm is o(1)}, it suffices to prove
\[
    \left|\frac{\partial\eta_{m}^{\ast}}{\partial x_{i}}(0)\right|\leq C(|d_{1,m}|+|d_{2,m}|),\quad \text{for}\;\; i=1,2.
\]
As in \eqref{eq:psi m equation}, we deduce
\[
-\Delta \eta_m^*
=
q_m^C \eta_m^*
+
\sum_{j=1}^2 d_{j,m}H_{j,m}^C,
\]
where $q_m^C := \lambda_m V(x)|x|^2 C_m(x)$, $H_{j,m}^C := q_m^C PZ_m^j - |x|^2 e^{W_m}Z_m^j$. 

We define $\widetilde{\eta}_m^*(y):=\eta_m^*(\rho_m y),$ by applying the change of variable $x=\rho_m y$, we obtain in \(B_2(0)\),
\[
-\Delta_y\widetilde{\eta}_m^*
=
\widetilde{q}_m^C(y)\widetilde{\eta}_m^*
+
\sum_{j=1}^2 d_{j,m}\widetilde{H}_{j,m}^C(y),
\]
where $\widetilde{q}_m^{C}(y):=\rho_m^2 q_m^{C}(\rho_m y)$, $\widetilde {H}_{j,m}^{C}(y):=\rho_{m}^{2} H_{j,m}^{C}(\rho_{m} y)$.

Analogously to the derivation of \eqref{eq:q m tuta Ls estimate}, we obtain
\[
   \|\widetilde{q}_m^C\|_{L^s(B_2)} \leq C_s \rho_m^{-2/s}.
\]

In what follows, we estimate $\widetilde H_{j,m}^{C}$ for $j=1,2$. We decompose $\widetilde H_{j,m}^{C}$ as
\[
\begin{aligned}
\widetilde{H}_{j,m}^{C}(y)
&= \rho_m^2 |\rho_m y|^2 e^{W_m(\rho_m y)}
\left( PZ_m^j(\rho_m y) - Z_m^j(\rho_m y) \right) \\
&\quad + \rho_m^2
\left( q_m^C(\rho_m y) - |\rho_m y|^2 e^{W_m(\rho_m y)} \right)
PZ_m^j(\rho_m y)\\
&=\widetilde{H}_{j,m}^{C,(1)}(y)+\widetilde{H}_{j,m}^{C,(2)}(y).
\end{aligned}
\]
Notice that $\widetilde{H}_{j,m}^{C,(1)}(y)=\widetilde{H}_{j,m}^{(1)}(y)$, one has
\begin{equation}\label{eq:Hjm C1 tuta Lp estimate}
\|\widetilde{H}_{j,m}^{C,(1)}(y)\|_{L^{p_0}(B_2)}\leq C\rho_{m}^{2}.
\end{equation}
For $\widetilde{H}_{j,m}^{C,(2)}(y)$, recall the definition of $Q_m(y)$, a direct computation yields that
\[
\begin{aligned}
\widetilde{H}_{j,m}^{C,(2)}(y)
&=
\left[
Q_m(y)E_m(y)F_m(\rho_m y)
-
\frac{Q_m(y)}{V(\rho_m y)}
\right]
PZ_m^j(\rho_m y) \\
&=
Q_m(y)PZ_m^j(\rho_m y)
\left[
E_m(y)F_m(\rho_m y)
-
\frac{1}{V(\rho_m y)}
\right],
\end{aligned}
\]
where $E_{m}(y)$ is defined in \eqref{eq:definition of Emy} and is bounded. Following the decomposition in \eqref{eq:decompose Emy e phim rhom y}, we write
\begin{equation}\label{eq:decompose Emy Fmy}
    E_m(y)F_m(\rho_my)- \frac{1}{V(\rho_m y)}=E_m(y)(F_m(\rho_my)-1)
+(E_m(y)-1)+\left(1-\frac{1}{V(\rho_m y)}\right).
\end{equation}
Recalling the definition $F_m$ in \eqref{eq:definition of Fmy}, by \eqref{eq:Psi m is o(1)}, H${\rm \ddot{o}}$lder's inequality and Lemma \ref{lem:M-T inequality}, we deduce
\begin{equation*}
\begin{aligned}
\|F_m(\rho_m\cdot)-1\|_{L^{p_0}(B_2)}
&\leq C\left(
\left\|e^{\phi_m^{(1)}(\rho_m\cdot)}-1\right\|_{L^{p_0}(B_2)}
+
\left\|e^{\phi_m^{(2)}(\rho_m\cdot)}-1\right\|_{L^{p_0}(B_2)}
\right.\\
&\qquad\left.
+
\left\|e^{\Psi_m(\rho_m\cdot)}-1\right\|_{L^{p_0}(B_2)}
\right),
\end{aligned}
\end{equation*}
where $p_0>2$. Using \eqref{eq: 1-e^phi Lp estimate}, \eqref{eq:e phim rmy Ls} and \eqref{eq:local uniqueness assumption}, we obtain
\begin{equation*}
    \|F_m(\rho_m\cdot)-1\|_{L^{p_0}(B_2)}\leq C(\rho_{m}^{-\frac{2}{p_0}}\lambda_{m}^{\frac{1}{2}-\varepsilon}+\lambda_{m}^{\sigma}).
\end{equation*}
Treating the last two terms on the right hand side of \eqref{eq:decompose Emy Fmy} in the same way as in Lemma \ref{lem:c1m=c2m=0}, we arrive at
\begin{equation}\label{eq:H jm2C tuta Lp estimate}
    \|\widetilde{H}_{j,m}^{C,(2)}(y)\|_{L^{p_0}(B_2)}\le C(\rho_{m}^{-2/p_0+2-4\varepsilon}+\rho_m^{4\sigma}).
\end{equation}
A combination of \eqref{eq:Hjm C1 tuta Lp estimate} and \eqref{eq:H jm2C tuta Lp estimate} leads to
\[
\|\widetilde{H}_{j,m}^{C}(y)\|_{L^{p_0}(B_2)}\le C(\rho_{m}^{-2/p_0+2-4\varepsilon}+\rho_{m}^{4\sigma}).
\]
Analogously to \eqref{eq:interior estimate for psi m star tuta}-\eqref{eq:nabla y psi m star tuta 0}, we get
\[
|\nabla_y \widetilde{\eta}_m^*(0)|
\le\|\nabla\widetilde{\eta}_{m}^{\ast}\|\leq
C (\rho_m^{1+4\sigma'-\frac{2}{p_0}}
+\rho_m^{4\sigma})\left(|d_{1,m}|+|d_{2,m}|\right),
\]
where $p_0>2$, $\sigma>\frac{1}{4}$. Scaling back, we obtain
\[
|\nabla_x \eta_m^*(0)|
\le|\rho_m^{-1}\nabla_y \widetilde{\eta}_m^*(0)|\le
C(\rho_m^{4\sigma'-\frac{2}{p_0}}+\rho_{m}^{4\sigma-1})\left(|d_{1,m}|+|d_{2,m}|\right).
\]
Since $\rho_m=\left(\frac{\lambda_m}{32}\right)^{1/4}$, we choose $p_0$ sufficiently large such that $4\sigma'-\frac{2}{p_0}>0$. Hence,
\[
|\nabla_x\eta_m^*(0)|
\le
C\left(|d_{1,m}|+|d_{2,m}|\right),
\]
which implies
\[
|\mathcal{Q}_{i,m}^{\text{loc}}(\eta_{m}^{\ast})|=o(1)(|d_{1,m}|+|d_{2,m}|),\quad i=1,2.
\]

We now focus on the main terms $\mathcal{Q}_{i,m}(PZ_m^j)$, $i,j=1,2.$ As in Lemma \ref{lem:c1m=c2m=0}, we  decompose  $\mathcal{Q}_{i,m}(PZ_m^j)$ as
\begin{equation}\label{eq:Qim PZmj decompose}
\mathcal{Q}_{i,m}(PZ_m^j)
=
\mathcal{Q}_{i,m}^{\mathrm{vol}}(PZ_m^j)
+
\mathcal{Q}_{i,m}^{\mathrm{bd}}(PZ_m^j)
+
\mathcal{Q}_{i,m}^{\mathrm{loc}}(PZ_m^j).
\end{equation}
Following the argument in Lemma \ref{lem:c1m=c2m=0}, we obtain
\begin{equation}\label{eq:Qim bd loc are o(1)}
    \mathcal{Q}_{i,m}^{\mathrm{bd}}(PZ_m^j)=\mathcal{Q}_{i,m}^{\mathrm{loc}}(PZ_m^j)=o(1), \quad i,j=1,2.
\end{equation}
Inserting \eqref{eq:Qim bd loc are o(1)} into \eqref{eq:Qim PZmj decompose} yields
\[\mathcal{Q}_{i,m}(PZ_m^j)=\mathcal{Q}_{i,m}^{\mathrm{vol}}(PZ_m^j)
+o(1), \quad i,j=1,2.\]

We first consider $\mathcal{Q}_{1,m}^{\mathrm{vol}}(PZ_m^j)$ for $j=1,2$. By \eqref{eq:V}, we have
\begin{equation*}
\mathcal{Q}_{1,m}^{\mathrm{vol}}(PZ_m^j)
=
\int_{B_1}
\lambda_m C_m PZ_m^j
\left(\gamma_1 x_1^2-\gamma_2 x_2^2\right)\,dx
+o(1).
\end{equation*}
By applying the change of variable  $x=\rho_{m}y$ and recalling the definition of  $\xi_m$, we deduce
\[
\mathcal{Q}_{1,m}^{\mathrm{vol}}(PZ_m^j)
= 32\int_{B_{1/\rho_m}}
\frac{(\gamma_1 y_1^2-\gamma_2 y_2^2)\zeta_{mj}}
{(1+|\zeta_m|^2)^3}E_m(y)F_{m}(\rho_my)\,dy
+o(1),
\]
where $F_m$ is defined in \eqref{eq:definition of Fmy} and $\zeta_m:=y^{2}-\xi_m$. Moreover, since both $\phi_{m}^{(1)}$ and $\phi_{m}^{(2)}$ satisfy the same smallness estimates as $\phi_m$, the argument used to prove \eqref{eq:Emy e phim-1} gives
\[
    \int_{B_{1/\rho_m}}
\frac{(\gamma_1 y_1^2-\gamma_2 y_2^2)\zeta_{mj}}
{(1+|\zeta_m|^2)^3}(E_m(y)F_m(\rho_my)-1)\,dy=o(1).
\]
Consequently,
\begin{equation*}
\mathcal{Q}_{1,m}^{\mathrm{vol}}(PZ_m^j)
= 32\int_{B_{1/\rho_m}}
\frac{(\gamma_1 y_1^2-\gamma_2 y_2^2)\zeta_{mj}}
{(1+|\zeta_m|^2)^3}\,dy
+o(1),\quad j=1,2.
\end{equation*}
Moreover, we derive
\[
\mathcal Q^{\rm vol}_{2,m}(PZ_m^j)
=
32(\gamma_1+\gamma_2)
\int_{B_{1/\rho_m}}
\frac{y_1y_2\zeta_{mj}}
{(1+|\zeta_m|^2)^3}\,dy
+o(1),
\qquad j=1,2.
\]
After the change of variables \(\zeta_m=y^2-\xi_m\), and using
\(\xi_m\to \xi=(\xi^{*},0)\in\mathbb{R}^{2}\), we deduce that
\[
\mathcal{Q}_{i,m}^{\mathrm{vol}}(PZ_m^j)=\mathcal{P}_{i,m}^{\mathrm{vol}}(PZ_m^j)+o(1),\quad i,j=1,2.
\]
The rest of the proof proceeds as in Lemma \ref{lem:c1m=c2m=0}, and we conclude that
$d_{1,m}=d_{2,m}=0$
\end{proof}

We now prove Theorem \ref{thm:local uniqueness}.
\begin{proof}
By Lemma \ref{lem:d1m=d2m=0} and \eqref{decompose eta m}, we obtain $\eta_m=\eta_{m}^{\ast}\in K_{m}^{\perp}$. Then Lemma \ref{lem:estimate for eta m star} shows that
\[
\|\eta_m^*\|
\leq
C\lambda_m^{\frac14+\sigma'}
(|d_{1,m}|+|d_{2,m}|)=0.
\]
Consequently, $\|\eta_m\|=0$, which contradicts $\|\eta_m\|=1$.
\end{proof}

\section*{Acknowledgments}
Yuxia Guo was supported by the National Natural Science Foundation of China (No. 12271283) and  National Key R\&D Program (2023YFA1010002).

\section*{Statements and Declarations}

The authors confirm that there are no relevant financial or non-financial competing interests to report.
	
\section*{Data Availability Statements}

All data generated or analyzed during this study are included in this article.

\end{document}